\documentclass[11pt]{amsart}
\usepackage{ae}
\usepackage[all]{xy}
\usepackage{graphicx,amsfonts,amssymb,amsmath}
\usepackage{natbib}
\usepackage{pifont}

 \usepackage{hyperref}
\hypersetup{colorlinks,linkcolor=blue,citecolor=red,}

\usepackage[OT2,T1]{fontenc}
\DeclareSymbolFont{cyrletters}{OT2}{wncyr}{m}{n}
\DeclareMathSymbol{\sha}{\mathalpha}{cyrletters}{"58}
\input cyracc.def

 \newtheorem{thm}{Theorem}[section]
 \newtheorem{ques}[thm]{Question}
 \newtheorem{ex}[thm]{Example}
 \newtheorem{cor}[thm]{Corollary}
 \newtheorem{lem}[thm]{Lemma}
 \newtheorem{prop}[thm]{Proposition}
 \theoremstyle{definition}
 \newtheorem{defn}[thm]{Definition}
 \theoremstyle{remark}
 
 \theoremstyle{remark}
 \newtheorem{rem}[thm]{Remark}
 \numberwithin{equation}{subsection}

 \newcommand{\F}{\mathbb{F}}
 \newcommand{\tor}{\textup{tor}}
 \renewcommand{\u}{\textup{u}}
 \newcommand{\red}{\textup{red}}
 \renewcommand{\ss}{\textup{ss}}
 \newcommand{\ssu}{\textup{ssu}}
 \newcommand{\E}{\mathcal{E}}
 
 \newcommand{\im}{\textup{Im}}
 \renewcommand{\ker}{\textup{Ker}}
 \newcommand{\Pic}{\textup{Pic}}
 \newcommand{\Br}{\textup{Br}}
 \renewcommand{\H}{\textup{H}}
 
 \newcommand{\Gal}{\textup{Gal}}
 
 \newcommand{\Hom}{\textup{Hom}}
 \renewcommand{\dim}{\textup{dim}}
  
 \newcommand{\BM}{Brauer\textendash Manin\ }
 \newcommand{\BMo}{Brauer\textendash Manin obstruction\ }
  \newcommand{\A}{\mathbb{A}}
  \newcommand{\codim}{\textup{codim}}

 \newcommand{\Q}{\mathbb{Q}}

 \newcommand{\Z}{\mathbb{Z}}
 \newcommand{\G}{\mathbb{G}}
 
 \renewcommand{\O}{\mathcal{O}}

\begin{document}

\title[Arithmetic purity]
 {Arithmetic purity of strong approximation for homogeneous spaces}

\author{Yang Cao}
\author{Yongqi Liang}
\author{Fei Xu}

\date{\today.}



\begin{abstract}
We prove that any open subset $U$ of a semi-simple simply connected quasi-split linear algebraic group $G$ with $\codim (G\setminus U, G)\geq 2$ over a number field satisfies strong approximation by establishing a fibration of $G$ over a toric variety. We also prove a similar result of strong approximation with \BMo for a partial equivariant smooth compactification of a homogeneous space where all invertible functions are constant and the semi-simple part of the linear algebraic group is quasi-split. Some semi-abelian varieties of any given dimension where the complements of a rational point do not satisfy strong approximation with \BM obstruction are given.
\end{abstract}

\maketitle

\tableofcontents


\section{Introduction}

It is well known that weak approximation over a number field is birationally invariant among smooth varieties by the implicit function theorem. Using Manin's idea, one can further study weak approximation with \BM obstruction. If a variety over a number field satisfies weak approximation with \BM obstruction which is conjectured to be true for rationally connected smooth varieties by Colliot-Th\'el\`ene (cf. \cite{CT03}), so are its open subsets with complement of codimension $\geq 2$ by the purity theorem for  \'etale cohomology.

Under a stronger topology -- the adelic topology -- instead of the product topology, one can study strong approximation which is related to integral points on a variety. Indeed,  Eichler  \citep{Eic38}-\citep{Eic52}, Weil \citep{Weil}, Shimura  \citep{Shi64}, Kneser  \citep{Kne65}, Platonov  \citep{Plat1}-\citep{Plat2}, Prasad  \citep{Pra77} and others established strong approximation for semi-simple simply connected linear algebraic groups. More recently, Browning and Schindler \cite{BS} established strong approximation for certain norm varieties by using analytic methods.
Min\v{c}hev in \citep{Minchev} pointed out that strong approximation cannot be true for varieties which are not simply connected. Therefore strong approximation is not birationally invariant among smooth varieties. However, one can expect that it is invariant among smooth varieties up to a closed sub-variety of codimension at least 2 as Zariski\textendash Nagata purity theorem. Indeed, Wei \cite[Lemma 1.1]{Wei14} and the first and the third author \cite[Proposition 3.6]{CaoXuToric} proved that this is true for the affine space ${\bf A}^{n}$ independently, which was applied to show strong approximation with Brauer\textendash Manin obstruction for toric varieties. Harpaz and Wittenberg gave another application of this result in \citep{HW}. Based on this evidence, Wittenberg proposed such an invariance problem in \cite[Problem 6]{AIM14} and his recent survey \citep[Question 2.11]{Wit16}. Colliot-Th\'el\`ene asked the same question in his talk in 2016 Indo-French conference in Chennai.  In this paper we give an affirmative answer for quasi-split semi-simple simply connected linear algebraic groups. In particular, we give a new proof of strong approximation for \emph{quasi-split} semi-simple simply connected groups by establishing fibrations over toric varieties.

 \begin{thm}[Theorem \ref{sss-quasi-split}]\label{maintheorem0}
Let $G$ be a semi-simple simply connected and quasi-split group over a number field $k$ and $S\neq \emptyset$ be a finite set of places of $k$. Then any open subset $U$ of $G$ with $\codim(G\setminus U, G)\geq 2$ satisfies strong approximation off $S$.
\end{thm}

In \citep{CTXu09},  Colliot-Th\'el\`ene and the third author first studied strong approximation with \BM obstruction and established strong approximation with \BM obstruction for homogeneous spaces of semi-simple simply connected groups. Since then, Harari \citep{Harari08}, Demarche \citep{Demarche11} and Wei and the third author \cite{WX} proved strong approximation with \BM obstruction for various linear algebraic groups. Borovoi and Demarche \citep{BD13} showed strong approximation with Brauer\textendash Manin obstruction for homogeneous spaces with geometrically connected stabilizer. Colliot-Th\'el\`ene and the third author \citep{CTXu13},  Colliot-Th\'el\`ene and Harari \citep{CTH16}, the third author \citep{Xu15} and Derenthal and Wei \cite{DW}  established strong approximation for certain families of homogeneous spaces. The first and the third author \cite{CaoXuToric}, \citep{CaoXuGroupic} proved strong approximation with \BM obstruction for a partial equivariant smooth compactification of a connected linear algebraic group and the first author \citep{Cao-homog}  further established strong approximation with \BM obstruction for a partial equivariant smooth compactification of a homogeneous space by improving the descent method with combination of the fibration method.

On the other hand, many examples do not satisfy strong approximation with \BM obstruction. We mention the following one which is closely related to the  topic of the present paper.
In \cite[Example 5.2]{CaoXuToric}, the first and the third author consider the $\Q$-variety $X=\G_{m}\times\G_a\setminus \{(1,0)\}$. They show that for $k=\Q$ or an imaginary quadratic field, $X_{k}$ does not satisfy strong approximation with \BMo off $\infty_{k}$. However, for all other number fields $k$, $X_{k}$ does  satisfy strong approximation with \BMo off $\infty_{k}$.
Such an example shows that the property of strong approximation  with \BMo
\begin{itemize}
\item[-] does not hold in general even for $k$-rational varieties;
\item[-] may become valid over some finite extension of the ground field;
\item[-] is not invariant under the removal of a codimension $\geq2$ closed subset.
\end{itemize}

 Keeping such an example in mind, we may ask a similar question for strong approximation with Brauer\textendash Manin obstruction.
\begin{ques}\label{mainquestion} Let $X$ be a smooth geometrically integral variety over a number field $k$ and $S$ be a finite set of places of $k$.
Suppose that  $\Pic(X_{\bar{k}})$ is finitely generated and $\bar{k}[X]^\times =\bar{k}^{\times}$ where $\bar{k}$ is an algebraic closure of $k$. If $X$ satisfies strong approximation with \BMo off $S$,  does any open sub-variety of $X$ with complement of codimension $\geq 2$ satisfies the same property?
\end{ques}

We call such phenomena \textbf{arithmetic purity} of strong approximation (with \BM obstruction) off $S$. D. Wei gave an affirmative answer to this question for smooth toric varieties in \cite[Theorem 0.2]{Wei14}. In the present paper, we extend this result to partial smooth equivariant compactifications of homogeneous spaces of connected linear algebraic groups.

\begin{thm}[Theorem \ref{purity-com-hom}]\label{maintheorem}
Let $G$ be a connected linear algebraic group, $X$ be a smooth and geometrically integral variety with an action of $G$ over a number field $k$ and $S\neq \emptyset$ be a finite set of places of $k$.  Suppose that $X$ contains a rational point with a connected stabilizer and a Zariski-open dense orbit. If $\bar{k}[X]^\times=\bar{k}^\times$ and a simply connected covering $G^{\textup{sc}}$ of the semi-simple part of $G$ over $k$ satisfies the arithmetic purity of strong approximation off $S$, then $X$ satisfies the arithmetic purity of strong approximation with \BMo off $S$.
\end{thm}

The purity assumption on $G^{\textup{sc}}$ holds when $G^{\textup{sc}}$ is quasi-split according to Theorem \ref{maintheorem0}.  

Furthermore, in \S \ref{examples} we produce some examples which do not verify arithmetic purity of strong approximation even for arbitrarily large codimension. These examples show that the geometric assumptions in Question \ref{mainquestion} are necessary. The proof is based on the combination of an extension of Poonen's argument \cite{Poonen} which was further explained and generalised by Colliot-Th\'el\`ene, P\'al and Skorobogatov in \cite{CTPSk} with an extension of Harari\textendash Voloch's argument \cite{HV10}. 

\begin{ex}[Theorem \ref{elliptic} and Corollary \ref{abelianvar}]
Let $E$ be a semi-abelian variety over a number field $k$, with $\dim(E)=1$ such that $E(k)$ is not discrete in $E(\A_k^{\infty_k})$. 

If $E$ is an elliptic curve over $k$, then the complement of any $k$-rational point in $E$ does not satisfy strong approximation with \BMo off $\infty_k$.

If $A$ is a non-trivial semi-abelian variety over $k$ such that $A(k)$ is discrete in $A(\A_k^{\infty_k})$, then the complement of any $k$-rational point in $A\times_k E$ does not satisfy strong approximation with \BMo off $\infty_k$.
\end{ex}

Our examples, viewed as fibrations over $A$, have the following two new features in contrast to Poonen's construction in \cite{Poonen}
\begin{itemize}
\item[-] the Brauer group of the total space does not come from the base variety;
\item[-] the \BMo does not control the failure of strong approximation on a ``bad'' fibre.
\end{itemize}
Compared to \cite[Example 5.2]{CaoXuToric} mentioned above, these examples provide some evidence that the property of strong approximation with \BMo may get lost over some finite field extensions.

It should be pointed out that Hassett and Tschinkel discussed the similar type of question about potential Zariski density of integral points in \cite[\S 3.3 and 5]{HT}.

\section{Notation and terminology}

In the present paper, the base field $k$ will be a number field if not otherwise specified. We denote by $\Omega_{k}$ the set of places of $k$ and denote by $\infty_{k}$ the set of Archimedean places. For each place $v\in\Omega_{k}$,  the local field $k_{v}$ is the completion of $k$ with respect to the absolute value of $v$ and $k_\infty=\prod_{v\in \infty_k} k_v$. We denote by $\O_{v}$ the ring of integers of $k_v$ for $v\in\Omega_{k}\setminus \infty_{k}$. The ring of ad\`eles  is denoted by $\A_{k}$ and the ring of ad\`eles without $S$-components is denoted by $\A_{k}^{S}$ for a finite set $S$ of $\Omega_k$. Moreover, we denote by $pr^{S}:\A_{k}\to\A_{k}^{S}$  the natural projection and the induced projections on adelic points of varieties. For any finite subset $S$ of $\Omega_k$ containing $\infty_k$, the ring of $S$-integers of $k$ is denoted by $\O_S$. When $S=\infty_k$, we simply write $\O_k$ for $\O_S$. 

For any abelian group $A$, we write $A^D=\Hom (A, \Bbb Q/\Bbb Z)$. For any ring $R$, we write $R^\times$ for the set of invertible elements with respect to multiplication of $R$.

Let $G$ be a connected linear algebraic group over $k$. Denote by $G^{\textup{u}}$ the unipotent radical of $G$. Then $G^{\textup{red}}=G/G^{\textup{u}}$ is a reductive group. The semi-simple part $G^{\textup{ss}}$ of $G$ is the commutator subgroup of $G^{\textup{red}}$. The quotient $G^{\textup{tor}}=G^{\textup{red}}/G^{\textup{ss}}$ is an algebraic torus. Let $G^{\ssu}=\ker(G\to G^{\tor})$ which is an extension of $G^{\ss}$ by $G^\u$. Let $G^{\textup{sc}}$ be a simply connected covering of $G^{\ss}$ over $k$. Fix an algebraic closure $\bar k$ of $k$, the character group of a (not necessarily connected) linear algebraic group $H$ is denoted by $H^*=\Hom_{{\bar k}-\textup{gp}}(H_{\bar k}, \Bbb G_{m,\bar{k}})$. Then $H^*=\bar{k}[H]^\times/\bar{k}^\times$ by Rosenlicht Lemma.

Let $X$ be an algebraic variety (separated geometrically connected scheme of finite type) over $k$ with Brauer group $\Br(X)=\H^{2}_{\scriptsize{\textup{\'et}}}(X,\G_{m})$. The Brauer group has a natural filtration given by  $$\Br_{1}(X)=\ker[\Br(X)\to\Br(X_{\bar{k}})] \ \ \ \text{and} \ \ \ \Br_{0}(X)=\im[\Br(k)\to\Br(X)] . $$
The Brauer\textendash Manin set $X(\A_{k})^{B}$ consists of adelic points that are orthogonal to a subgroup $B\subset\Br(X)$ under the Brauer\textendash Manin pairing. It is a closed subset of $X(\A_{k})$ with the adelic topology and contains the diagonal image of the set of rational points $X(k)$ by class field theory. When $B=\Br(X)$, we simply write $X(\A_{k})^{\Br}$ for $X(\A_{k})^{\Br(X)}$. For a finite set $S$ containing $\infty_k$, an integral model $\mathcal{X}$ of $X$ is defined to be a separated scheme of finite type over $\O_S$ satisfying $\mathcal{X}\times_{\O_S} k=X$. For any extension $K$ of $k$, we write $X_{K}=X\otimes_{k}K$ and denote by $K[X]$ the ring of regular functions of $X_K$ over $K$. An algebraic variety $X$ with an action of a linear algebraic group $G$ is called a $G$-variety.

\begin{defn} \label{def-sa} Let $X$ be a variety over a number field $k$, $S$ be a finite subset of $\Omega_k$ and $B$ be a subgroup of $\Br(X)$.

We say that $X$ satisfies \emph{strong approximation off $S$} if the diagonal image of $X(k)$ is dense in $pr^S (X(\A_{k}))$.

We say that  $X$ satisfies \emph{strong approximation with respect to $B$ off $S$} if the diagonal image of $X(k)$ is dense in $pr^{S}(X(\A_{k})^{B})$. When $B=\Br(X)$, we say $X$ satisfies \emph{strong approximation with \BM obstruction off $S$}.

We say that $X$ satisfies \emph{Zariski open strong approximation with respect to $B$ off $S$} if the diagonal image of $U(k)$ is dense in $pr^{S}(X(\A_{k})^{B})$ for any Zariski open dense subset $U$ of $X$. When $B=\Br(X)$, we say that $X$ satisfies \emph{Zariski open strong approximation with \BMo off $S$}.
\end{defn}

Let $X$ be a smooth and geometrically integral variety over a number field $k$ and let $i: U\hookrightarrow X$ be a Zariski open dense subset over $k$. Fix a finite subset $S$ of $\Omega_k$, we consider the following three statements related to Definition \ref{def-sa}.

(A1) $X(k)$ is dense in $pr^S (X(\A_{k}))$. 

(A2) $U(k)$ is dense in $pr^S (X(\A_{k}))$.

(A3) $U(k)$ is dense in $pr^S (U(\A_{k}))$.

Then

$$ ({\rm A1}) \Longleftrightarrow ({\rm A2})  \Longleftarrow ({\rm A3}) $$ by \cite[Lemma 3.2]{PR}. It is obvious that $({\rm A2}) \not \Rightarrow (\rm{A3})$, for example $X=\Bbb P^1$ and $U=\Bbb G_m$.

Similarly, one also has the following three statements with \BM obstruction for a subgroup $B$ of $\Br(X)$.

(B1)  $X(k)$ is dense in $pr^S (X(\A_{k})^B)$. 

(B2) $U(k)$ is dense in $pr^S (X(\A_{k})^B)$.

(B3) $U(k)$ is dense in $pr^S (U(\A_{k})^B)$.

In general, one only has $$({\rm B2}) \Longrightarrow ({\rm B1}). $$ 

The example for $({\rm A2}) \not \Rightarrow (\rm{A3})$ also shows that $({\rm B2}) \not \Rightarrow (\rm{B3})$ when $B=\Br(X)$.

 The statement (B1) does not imply (B2). For example, Let $X$ be an elliptic curve over a number field $k$. Then $X$ satisfies strong approximation with \BMo off $\infty_k$ if $\sha_k (X)$ is finite. Suppose that $X(k)$ is finite. Let $U=X\setminus X(k)$. Then $U$ cannot satisfy (B2) when $S=\infty_k$ and $B=\Br(X)$.

The statement (B3) does not imply (B1) either. For example, let $E$ be an elliptic curve over a number field $k$ such that both $E(k)$ and $\sha_k(E)$ are finite. Suppose that $E(k)$ contains more than one element. Fix $e\in E(k)$ and $p\in \Bbb G_a(k)$. Let 
$$X=(E\times_k \Bbb G_a) \setminus \{(e,p)\}  \ \ \ \text{and} \ \ \ U=(E\setminus \{e\})\times_k \Bbb G_a .$$ Then $U$ is an open dense subset of $X$. Since $\Br(E\setminus \{e\})=\Br(E)$ by \cite[III, ex.2.22 a)]{Milne80} and \cite[Theorem 6.4.4]{GS06}, one has 
$$\Br(U)=\Br(E\setminus \{e\})=\Br(E) =\Br(E\times_k \Bbb G_a)= \Br(X) $$ by \cite[Lemma 2.1]{CDX} and the cohomological purity. Then $U$ satisfies strong approximation with respect to $\Br(X)$ off $\infty_k$ by \cite[Proposition 3.1 and 3.2]{LiuXu15}. On the other hand, consider the fibration $\pi: X \rightarrow E$ obtained by the projection. Since $\pi^{-1}(e)\cong \Bbb G_m$ does not satisfy strong approximation off $\infty_k$, one concludes that $X$ does not satisfy strong approximation with \BMo off $\infty_k$ by \cite[Lemma 5.1]{CaoXuToric}.  

As a consequence, $({\rm B3}) \not \Rightarrow (\rm{B2})$.

However, when $B$ is finite, one has 
$$ ({\rm B1}) \Longleftrightarrow ({\rm B2})  \Longleftarrow ({\rm B3}) , $$ where the first equivalence is given by the following proposition and the second implication follows from Harari's formal lemma (see \citep[Proposition 2.6]{CTXu13}).

\begin{prop} \label{ndiff} Let $X$ be a smooth and geometrically integral variety over a number field $k$, $B$ be a finite subgroup of $\Br(X)$ and $S$ be a finite set of $\Omega_k$. If $X$ satisfies strong approximation with respect to $B$ off $S$, then $X$ satisfies Zariski open strong approximation with respect to $B$ off $S$.
\end{prop}
\begin{proof}
For a Zariski open dense subset  $U$ of $X$, there are a finite subset $P_0$ of $\Omega_k$ containing $S\cup \infty_k$ and an integral model $\mathcal{X}$ of $X$ over $\O_{P_0}$ such that $\mathcal{X}(\O_v) \perp B$ for all $v\not \in P_0$. For any open subset
$$W= \prod_{v\in P}W_{v}\times\prod_{v\notin P}\mathcal{X}(\O_{v})\subset X(\A_{k})$$ with $P\supset P_{0}$ and $W_v= X(k_v) $ for $v\in S$ such that $W\cap X(\A_k)^{B}\neq \emptyset$, one can choose $v_0\not \in P$ and define
$$ W_1= \prod_{v\in P}W_{v}\times (U(k_{v_0})\cap \mathcal{X}(\O_{v_0}))\times \prod_{v\notin (P\cup \{v_0\})}\mathcal{X}(\O_{v}) \subset W .$$ Since $X$ is smooth over $k$, then $U(k_{v_0})$ is dense in $X(k_{v_0})$ by \cite [Lemma 3.2]{PR}. This implies that $$U(k_{v_0})\cap \mathcal{X}(\O_{v_0})\neq \emptyset \ \ \ \text{and} \ \ \  W_1\cap X(\A_k)^{B}\neq \emptyset . $$ Therefore
$ U(k) \cap W_1 = X(k) \cap W_1  \neq \emptyset $ as required.
\end{proof}

 In Section \ref{Zosa}, we will give a complete description of Zariski open strong approximation with Brauer\textendash Manin obstruction for connected linear algebraic groups.

Under the assumption of Question \ref{mainquestion} with $[B:\Br_1(X)] < \infty$, we also have
$$ ({\rm B1}) \Longleftrightarrow ({\rm B2})  \Longleftarrow ({\rm B3}) $$ by the following result.

\begin{prop} \label{ext}  Let $X$ be a smooth and geometrically integral variety over a number field $k$ with $\bar{k}[X]^\times=\bar{k}^\times$ and $S$ be a finite subset of $\Omega_k$. 
Suppose that $\Pic(X_{\bar{k}})$ is finitely generated and $$\Br_{1}(X)\subseteq B\subseteq \Br(X) \ \ \ \text{such that} \ \ \  [B:\Br_1(X)] < \infty .$$

(1) If a Zariski open dense subset $U$ of $X$ satisfies strong approximation with respect to $B$ off $S$, then $X$ satisfies strong approximation with respect to $B$ off $S$.

(2) If $X$ satisfies strong approximation with respect to $B$ off $S$, then $X$ satisfies Zariski open strong approximation with respect to $B$ off $S$.

\end{prop}

\begin{proof}
One can assume that $X(\A_{k})^{B}\neq\emptyset$. Let $(x_{v})\in X(\A_{k})^{B}$. There is a universal torsor $Y\xrightarrow{f} X$ with $(y_v)\in Y(\A_{k})$ such that $$f((y_v))=(x_v)$$ by \cite[Theorem 3]{Sk-beyond}.

Let $U$ be a Zariski open dense subset of $X$ and  $$V=U\times_{X} Y \ \ \ \text{ and } \ \ \ \Lambda\subset \Br(X)$$ be a finite set of representatives of $B/\Br_{1}(X)$. Let $P_0$ be a sufficiently large finite set containing $S$ and $\infty_k$ such that
\begin{itemize}
\item[(a)] The morphism $Y\xrightarrow{f} X$ extends to a morphism $\mathcal{Y}\xrightarrow{f} \mathcal{X}$ of integral models over $\O_{P_0}$ and $y_v\in \mathcal{Y}(\O_v)$ for $v\not\in P_0$.
\item[(b)] Both open immersions $U \rightarrow X$ and $V\rightarrow Y$ extend to open immersions $\mathcal{U} \rightarrow \mathcal{X}$ and $\mathcal{V} \rightarrow \mathcal{Y}$ of integral models over $\O_{P_0}$ respectively.
\item[(c)] $\mathcal{V}=\mathcal{U}\times_{\mathcal{X}} \mathcal{Y}$ over $\O_{P_0}$ and the projection $\mathcal{V}=\mathcal{U}\times_{\mathcal{X}} \mathcal{Y} \rightarrow \mathcal{U}$  is an extension of $V\xrightarrow{f|_{V}} U$.
\item[(d)] $\mathcal{X}(\O_v) \perp \Lambda$ and $\mathcal{V}(\O_v)\neq \emptyset$ by the Lang\textendash Weil estimates for all $v\not \in P_0$.
\end{itemize}

Let $$W= \prod_{v\in P}W_{v}\times\prod_{v\notin P}\mathcal{X}(\O_{v})\subset X(\A_{k})$$ be an open subset such that $W^B\neq \emptyset$, where $P\supset P_{0}$ and $W_v= X(k_v) $ for $v\in S$. Take $(x_{v})\in W^B$.  For all $v\in P$, 
we choose $y_v'\in V(k_v)\cap f^{-1}(W_v)$ sufficiently close to $y_v$ by \cite [Lemma 3.2]{PR} such that
$$ b(f(y_v'))=b(f(y_v))=b(x_v) $$ for all $b\in \Lambda$.

For $v\not\in P$, we choose $y_v'\in \mathcal{V}(\O_v)$ by (d). Then $(y_v')\in V(\A_k) \subseteq Y(\A_k)$ with $$f((y_v'))\in W\cap U(\A_k)^{\Br_1(X)}$$ by \cite[Theorem 3]{Sk-beyond}. For $v\not\in P$, both $f(y_v')$ and $x_v$ are in $\mathcal{X}(\O_v)$ by (a), (b) and (c). This implies that
$b(f(y_v'))=b(x_v)$ are trivial for all $b\in \Lambda$ and $v\not\in P$ by (d). Therefore
$f((y_v'))\in W\cap U(\A_k)^{B} $. 

For case (1), we have
 $$(X(k)\cap W)\supseteq (U(k)\cap W) \neq \emptyset $$ as desired by the assumption on $U$.
 
 For case (2), we choose $v_0\not\in P$ and define 
 $$W_0 = (\prod_{v\in P}W_{v})\times f(\mathcal{V}(\O_{v_0})) \times (\prod_{v\notin P\cup \{v_0\} }\mathcal{X}(\O_{v})) \subset  W.$$ Then $W_0$ is an open subset of $X(\A_k)$ since $f$ is smooth. Therefore $W_0^B\neq \emptyset$ by the above argument. Then 
 $$ X(k)\cap W_0 =U(k) \cap W_0 \neq \emptyset $$ as desired by the assumption on $X$. 
\end{proof}

\begin{defn} Let $X$ be a smooth variety over a number field $k$, $B$ be a subgroup of $\Br(X)$, $S$ be a finite subset of $\Omega_k$ and $c$ be an integer bigger than 1.

We say that $X$ satisfies \emph{the arithmetic purity off $S$ of codimension $c$} if any open subvariety with complement of codimension $\geq c$ satisfies strong approximation off $S$. In particular, when $c=2$, we simply say that $X$ satisfies \emph{the arithmetic purity off $S$}.

We say that $X$ satisfies \emph{the arithmetic purity with respect to $B$ off $S$ of codimension $c$} if any open subvariety with complement of codimension $\geq c$ satisfies strong approximation with respect to $B$ off $S$. In particular, when $c=2$, we simply say that $X$ satisfies \emph{the arithmetic purity with respect to $B$ off $S$}.


\end{defn}

It is clear that the similar arithmetic purity for weak approximation with \BMo holds for smooth varieties by the cohomological purity theorem.

\begin{rem} In this paper we restrict ourselves to varieties over number fields. If $X$ satisfies strong approximation over a number field $k$ and $X(\A_k)\neq \emptyset$, then $X_{\bar{k}}$ is simply connected  by \cite[Theorem 1]{Minchev} (see also \cite[Proposition 2.2]{Rapinchuk}). The analogue is not true over a global function field. For example,  $\Bbb G_a$ satisfies strong approximation off a finite non-empty set of primes over a global function field  but $\Bbb G_a$ is not simply connected over an algebraic closure of this field. This is because $\Bbb G_a$ admits Artin\textendash Schreier \'etale coverings over a field of positive characteristic.
\end{rem}

\section{Arithmetic purity for semi-simple simply connected quasi-split linear algebraic groups} \label{qs-tv}
In this section we study arithmetic purity for semi-simple simply connected quasi-split linear algebraic groups by establishing fibrations over toric varieties.

Let $G$ be a quasi-split semi-simple simply connected linear algebraic group and $B$ be a Borel subgroup of $G$ over a field $k$ of characteristic 0. By the Bruhat decomposition (see \citep[(8.3.11) Corollary, (8.2.4) Proposition (ii) and (8.3.2) Lemma (ii)]{Sp}), there is a root $w_0\in G(k)$ (which is the longest one) such that the big cell $V_0=Bw_0B$ is a Zariski open dense subset of $G$ over $k$ by Galois descent.
If $B^{\u}$ is the unipotent radical of $B$ and $T$ is a maximal torus of $B$ over $k$, one has the split short exact sequence
$$ 1\rightarrow B^{\u} \rightarrow B \xrightarrow{p} T \rightarrow 1$$ of algebraic groups over $k$ where the map $p$ is $T$-equivariant. Moreover, one has an isomorphism of varieties over $k$
 \begin{equation} \label{isom}
B\times_k B^{\u} \xrightarrow{\cong} V_0 ; \ \ \ (b,u) \mapsto b\cdot w_0 \cdot u
\end{equation}   by \cite[(8.3.5) Lemma (ii) and (8.3.6) Lemma (ii)]{Sp}.  Combining (\ref{isom}) with the map $p$, one obtains a $T$-equivariant morphism over $k$
\begin{equation} \label{tor}
 \phi: \ V_0 \rightarrow T \ \ \ \ \ \text{with} \ \ \ \ \ \bar{k}[T]^\times/\bar{k}^\times \xrightarrow{\cong} \bar{k}[V_0]^\times/\bar{k}^\times
\end{equation}
 as $\Gal(\bar{k}/k)$-modules by (\ref{isom}).

Recall that a torus $T$ over $k$ is called quasi-trivial if $\bar{k}[T]^\times/\bar k^\times$ is a permutation $\Gal(\bar{k}/k)$-module, i.e.  $T={\rm{Res}}_{K/k}(\Bbb G_m)$ for some finite \'etale $k$-algebra $K$.

\begin{lem} \label{div} Suppose $G$ is a quasi-split semi-simple simply connected linear algebraic group over a field $k$ of characteristic 0 and $B$ is a Borel subgroup of $G$ over $k$. If $V_0=Bw_0B$ is the big cell of $G$ over $k$ as above, then the map
$$  \bar{k}[V_0]^\times/\bar{k}^\times \xrightarrow{\cong} {\rm{Div}}_{G_{\bar k} \setminus (V_0)_{\bar k}}(G_{\bar{k}}) ; \ \ \ f \mapsto div(f) $$
is a canonical isomorphism of $\Gal(\bar{k}/k)$-modules. Moreover, any maximal torus of $B$ over $k$ is quasi-trivial. \end{lem}
\begin{proof}  Since $G$ is semi-simple and simply connected, one has $\bar{k}[G]^\times =\bar{k}^\times$ and $\Pic(G_{\bar k})=0$ by \cite[Proposition 2.6]{CTXu09}.  These imply the first part. The second part follows from (\ref{tor}).    \end{proof}

\begin{rem} Let $G$ be the group of norm 1 elements of the division quaternion algebra over $\Bbb R$. It is well known that $G$ is semi-simple simply connected but not quasi-split over $\Bbb R$. Since the set of real points of any maximal torus of $G$ is compact but the set of real points of a quasi-trivial torus is not compact, one concludes that there are no quasi-trivial maximal tori of $G$ over $\Bbb R$.
\end{rem}

Since $T$ is quasi-trivial, there is a finite \'etale $k$-algebra $K$ such that $T={\rm{Res}}_{K/k}(\Bbb G_m)$. Then the morphism $\phi$ in (\ref{tor}) can be extended uniquely to a $T$-equivariant morphism
$$ \phi: \ G \rightarrow {\rm{Res}}_{K/k}(\Bbb G_a) $$ by Lemma \ref{div} and \citep[Proposition 2.3]{Cao-homog}. The following definition of standard toric varieties already appeared in \cite[Definition 3.11]{CaoXuGroupic}.

\begin{defn}\label{st-toric} Let $K/k$ be a finite \'etale $k$-algebra. The unique minimal open toric sub-variety $R$ of ${\rm{Res}}_{K/k}(\Bbb G_a)$ over $k$ with respect to ${\rm{Res}}_{K/k}(\Bbb G_m)$ such that
$$ \codim({\rm{Res}}_{K/k}(\Bbb G_a)\setminus R, {\rm{Res}}_{K/k}(\Bbb G_a))\geq 2 $$
is called the standard toric variety of ${\rm{Res}}_{K/k}(\Bbb G_a)$.
\end{defn}

The following proposition is a variant of \cite[Proposition 3.12]{CaoXuGroupic}.

\begin{prop} \label{toric-fibration} With the above notations, there is an open $T$-subvariety $Y$ of $G$ containing $V_0$ such that the extension $Y \xrightarrow{\phi}  R$ of $\phi$ in (\ref{tor})
is smooth with non-empty geometrically integral fibres.
\end{prop}
\begin{proof}  This follows from \cite[Proposition 2.2]{Cao-homog}.
\end{proof}

The following proposition should be well known. We provide the proof for completeness.

\begin{prop}\label{fiber-codim} Let $f:Y\to X$ be a faithfully flat morphism between geometrically integral varieties over a field $k$ and $c$ be a positive integer.
If $V$ is an open subset of $Y$ such that $\codim(Y\setminus V, Y)\geq c$, then there is an open dense subset $U$ of $X$ such that $$\codim(f^{-1}(x)\setminus V, f^{-1}(x))\geq c$$ for any closed point $x\in U$.
\end{prop}
\begin{proof} Consider $Z=Y\setminus V$ as a reduced closed sub-scheme of $Y$. Let $\{Z_i\}_{1\leq i\leq l}$ be the set of all irreducible components of $Z$ and $f_i=f|_{Z_i}$ for $1\leq i\leq l$.
For any given $1\leq i\leq l$, we consider the following two cases.

If $\overline{f_i(Z_i)}=X$, there is an open dense subset $W_i$ of $Z_i$ such that $f_i : W_i\to X$ is flat by \citep[11.1.1]{EGAIV}. Let $U_i=f_i(W_i)$. Then $U_i$ is an open dense subset of $X$. For any closed point $x\in  U_i$, one concludes that
\begin{equation}\label{dim-inequ} \dim(f_i^{-1}(x))= \dim (Z_i) -\dim(X) \leq  \dim(Z) - \dim (X) . \end{equation}

Otherwise $f_i (Z_i)$ is not dense in $X$. Then $U_i=X\setminus \overline{f_i(Z_i)} $ is an open dense subset of $X$ and $f_i^{-1}(x)$ is empty for any closed point $x\in U_i$.

Let $$U=\bigcap_{i=1}^l U_i . $$  It is clear that $U$ is an open dense subset of $X$ and
$$ \codim(f^{-1}(x)\setminus V, f^{-1}(x))= \dim(f^{-1}(x)) - \dim(\bigcup_{i=1}^l f_i^{-1}(x)) \geq c$$ by (\ref{dim-inequ}) for any closed point $x\in U$ as desired.
\end{proof}

The main result of this section is the following arithmetic purity theorem for quasi-split semi-simple simply connected groups.

 \begin{thm}\label{sss-quasi-split}
Let $G$ be a semi-simple simply connected and quasi-split group over a number field $k$ and $S$ be a non-empty finite subset of $\Omega_k$.  Then any open subset $U$ of $G$ with $\codim(G\setminus U, G)\geq 2$ satisfies strong approximation off $S$.
\end{thm}

\begin{proof}  Since $G$ is quasi-split, there is a Borel subgroup $B$ of $G$ over $k$. Let $V_0=Bw_0B$ be the big cell of $G$ where $w_0\in G(k)$ is a longest root. Suppose $T$ is a maximal torus of $B$. Then $T={\rm{Res}}_{K/k}(\G_m)$ for some \'etale $k$-algebra $K$ by Lemma \ref{div}. Let $R$ be the standard toric variety of ${\rm{Res}}_{K/k}(\G_a)$ in sense of (\ref{st-toric}). Then there are an open variety $Y$ of $G$ containing $V_0$ and a $T$-equivariant smooth surjective morphism $\phi: Y\rightarrow R$ with geometrically integral fibres which extends (\ref{tor}) by Proposition \ref{toric-fibration}.

Without loss of generality,  we assume that $U\subseteq Y$ with $\codim(Y\setminus U, Y)\geq 2$. In order to prove $U$ satisfies strong approximation off $S$,  one only needs to verify three conditions of \citep[Proposition 3.1]{CTXu13}.  Let $R_1=\phi(U)$. Then $R_1$ is an open subset of $R$ with $\codim(R\setminus R_1, R)\geq 2$ by the faithful flatness of $\phi$. Therefore $R_1$ satisfies strong approximation off $S$ by \citep[Proposition 3.6]{CaoXuToric} and Condition (i) of \citep[Proposition 3.1]{CTXu13} is satisfied. The rest of argument is to find an open dense subset $R_2$ of $R_1$ such that Condition (ii) and (iii) of \citep[Proposition 3.1]{CTXu13} are satisfied.

By Proposition \ref{fiber-codim}, there is an open dense subset $C$ of $R$ such that $$\codim(\phi^{-1}(x) \setminus U, \phi^{-1}(x))\geq 2 $$ for any $x\in C(k)$.
Choose $R_2=R_1\cap T\cap C$. Then $R_2$ is an open dense subset of $R$. For any field extension $K/k$ and  $x\in R_2(K)$, one has $\phi^{-1}(x)$ is isomorphic to an affine space over $K$ by (\ref{isom}) and (\ref{tor}).  Therefore $\phi^{-1}(x) \cap U$ satisfies strong approximation off $S$ for any $x\in R_2(k)$ by \citep[Proposition 3.6]{CaoXuToric} and $(\phi^{-1}(y) \cap U)(k_v)\neq \emptyset$ for any $y\in R_2(k_v)$ with $v\in S$. Conditions (ii) and (iii) of \citep[Proposition 3.1]{CTXu13} hold as desired.  \end{proof}

\section{Some fibration arguments with \BM obstruction}

In this section, we modify the easy fibration method in \cite[Proposition 3.1]{CTXu13} for strong approximation with \BM obstruction and extend the argument in \cite[\S 5]{CaoXuToric}.

\begin{prop}\label{fibration1}
Let $f:Y\to X$ be a smooth surjective morphism with geometrically integral fibres between smooth quasi-projective geometrically integral varieties over a number field $k$.
 Assume that $B$ is a subgroup of $\Br(X)$, $U$ is a Zariski open dense subset of $X$ and $S$ is a finite set of $\Omega_k$ such that
\begin{itemize}
\item[(1)] $U(k)$ is dense in $pr^{S}([\prod_{v\in S} f(Y(k_v)) \times X(\A_k^S)]^B)$; 
\item[(2)] For any $x\in U(k)$, the fibre $Y_x$ of $f$ over $x$ satisfies strong approximation off $S$.
\end{itemize}
Then $Y$ satisfies strong approximation with respect to $f^*(B)$ off $S$ where $f^*:  \Br(X)\to \Br(Y)$ is induced by $f$.
\end{prop}

\begin{proof}
Let $P_{0}$ be a sufficiently large finite subset of $\Omega_{k}$ containing $\infty_k$ and $S$  and $\mathcal{Y} \xrightarrow{f} \mathcal{X}$ be a morphism of integral models over $\O_{P_0}$ which extends $Y\xrightarrow{f} X$ such that
\begin{equation} \label{sur-int} f: \mathcal{Y}(\O_{v})\to\mathcal{X}(\O_{v}) \end{equation}
 is surjective for all $v\notin P_{0}$ by combining some standard results from \citep[9]{EGAIV} and the Lang\textendash Weil estimates.

For any finite subset $P\supset P_{0}$ of $\Omega_k$ and an open subset of $Y(\A_k)$
$$N=\prod_{v\in S}Y(k_{v})\times\prod_{v\in P\setminus S}N_{v}\times\prod_{v\notin P}\mathcal{Y}(\O_{v})$$
with $N\cap Y(\mathbb{A})^{f^*(B)}\neq\emptyset$, one concludes that the open subset
 $$M=\prod_{v\in S} f(Y(k_{v}))\times\prod_{v\in P\setminus S}f(N_{v})\times\prod_{v\notin P}\mathcal{X}(\mathcal{O}_{v})$$  of $X(\A_k)$ satisfies $M\cap X(\mathbb{A})^B\neq\emptyset $ by the functoriality of \BM pairing. By (1), there exists $x\in U(k)\cap M$.

Let $Y_x$ be the fibre of $f$ over $x$. Then the following open subset of $Y_{x}(\A_k)$
$$L=\prod_{v\in S}Y_{x}(k_{v})\times (\prod_{v\in P\setminus S}N_{v}\cap Y_{x}(k_{v}))\times\prod_{v\notin P}\mathcal{Y}_{x}(\mathcal{O}_{v}) \neq \emptyset $$
by (\ref{sur-int}). This implies that there exists $$y\in (Y_{x}(k)\cap L) \subset (Y(k) \cap N) $$ by (2) as required.
\end{proof}

\begin{cor} \label{cofib1}
Let $f:Y\to X$ be a smooth surjective morphism with geometrically integral fibres between smooth quasi-projective geometrically integral varieties over a number field $k$, $B$ be a subgroup of $\Br(X)$ and $S$ be a finite set of $\Omega_k$. Suppose that
\begin{itemize}
\item[(1)] $X$ satisfies Zariski open strong approximation with respect to $B$ off $S$;
\item[(2)] There is a Zariski open dense subset $U$ of $X$ such that the fibre $Y_x$ of $f$ over $x$ satisfies strong approximation off $S$ for all $x\in U(k)$;
\item[(3)] For any $v\in S$, $Y(k_{v})\xrightarrow{f} X(k_{v})$ is surjective.
\end{itemize}
Then $Y$ satisfies Zariski open strong approximation with respect to $f^*(B)$ off $S$ where $f^*:  \Br(X)\to \Br(Y)$ is induced by $f$.
\end{cor}

\begin{proof} For any Zariski open dense subset $V$ of $Y$, one has that $U_1=U\cap f(V)$ is a Zariski open dense subset of $X$ since $f$ is smooth. For any open subset $N$ of $Y(\A_k)$ with $N\cap Y(\mathbb{A})^{f^*(B)}\neq\emptyset$ as the proof of Proposition \ref{fibration1}, one obtains the same open subset $M$ of $X(\A_k)$ as above with $M\cap X(\mathbb{A})^B\neq\emptyset$. By (1) and (3), there is $x\in U_1(k)\cap M$. Since
$$L=\prod_{v\in S}Y_{x}(k_{v})\times (\prod_{v\in P\setminus S}N_{v}\cap V_{x}(k_{v}))\times\prod_{v\notin P}\mathcal{Y}_{x}(\mathcal{O}_{v}) \neq \emptyset $$
by (\ref{sur-int}), there exists $y\in (V_{x}(k)\cap L) \subset (V(k) \cap N)$ by (2) and Proposition \ref{ndiff} as required.
\end{proof}

One can also expect that \cite[Proposition 3.2]{LiuXu15} holds for Zariski open strong approximation with Brauer\textendash Manin obstruction.

\begin{prop}\label{product} Let $X$ and $Y$ be smooth and geometrically integral varieties over $k$ and $S$ be a finite subset of $\Omega_k$. Suppose that
$$ (X\times_k Y)(\A_k)^{\Br_1(X\times_k Y)} \neq \emptyset  . $$
Then $X\times_k Y$ satisfies Zariski open strong approximation with respect to $\Br_1(X\times_k Y)$ off $S$ if and only if $X$ and $Y$ satisfy Zariski open strong approximation with respect to $\Br_1(X)$ and $\Br_1(Y)$ off $S$ respectively.
\end{prop}
\begin{proof} Since both $X$ and $Y$ are geometrically integral varieties over $k$, one concludes that $X\times_k Y$ is also geometrically integral. 

Assuming $X$ and $Y$ satisfy Zariski open strong approximation with respect to $\Br_1(X)$ and $\Br_1(Y)$ off $S$, one wishes to prove $X\times_k Y$ satisfies Zariski open strong approximation with respect to $\Br_1(X\times_k Y)$ off $S$. Let $W$ be a Zariski open dense  subset of $X\times_k Y$. Since faithful flatness is stable under base change, one concludes that the projections $$p_X: X\times_k Y\rightarrow X \ \ \ \text{ and } \ \ \ p_Y: X\times_k Y\rightarrow Y$$ are faithfully flat. This implies that $p_X(W)$ and $p_Y(W)$ are Zariski open dense subsets of $X$ and $Y$ respectively. Since $\{M_\alpha\times N_\alpha\}_{\alpha\in I}$ forms an open basis of $(X\times_k Y)(\A_k)$ where $M_\alpha$ and $N_\alpha$ run over all open subsets of $X(\A_k)$ and $Y(\A_k)$ respectively, one only needs to show that $W(k)\cap (M\times N)\neq \emptyset$ as long as an open subset
$$M\times N= [(\prod_{v\in S} X(k_v))\times (\prod_{v\not\in S} M_v)]\times [(\prod_{v\in S} Y(k_v))\times (\prod_{v\not\in S} N_v)] \subseteq (X\times_k Y)(\A_k) $$
satisfies $$(M\times N)\cap [(X\times_k Y)(\A_k)^{\Br_1(X\times_k Y)}]\neq \emptyset$$ where $M_v$ and $N_v$ are open subsets of $X(k_v)$ and $Y(k_v)$ respectively. Since $$M\cap X(\A_k)^{\Br_1(X)}=p_X(M\times N)\cap X(\A_k)^{\Br_1(X)}\neq \emptyset $$ by functoriality of Brauer\textendash Manin pairing, there is $x_0\in p_X(W)(k)\cap M$ by the assumption on $X$. This implies that $W\cap (\{x_0\}\times_k Y)$ is a Zariski open dense subset of $\{x_0\}\times_k Y$.  
Since
$$N\cap Y(\A_k)^{\Br_1(Y)}=p_Y(M\times N)\cap Y(\A_k)^{\Br_1(Y)}\neq \emptyset $$ by functoriality of Brauer\textendash Manin pairing, one has  
$$(\{x_0\}\times N) \cap (\{x_0\} \times Y(\A_k)^{\Br_1(Y)})\neq \emptyset . $$
Therefore $$(W\cap (\{x_0\}\times_k Y))(k) \cap (\{x_0\} \times N)\neq \emptyset$$ as desired by the assumption on $Y$.

Conversely, assuming $X\times_k Y$ satisfies Zariski open strong approximation with respect to $\Br_1(X\times_k Y)$ off $S$, one has $(X\times_k Y)(k)\neq \emptyset$ by our assumption about non-emptyness of the Brauer\textendash Manin set. Fix $y_0\in Y(k)$ with
$$ i_{y_0}: X \xrightarrow{id_X \times y_0} X\times_k Y . $$ Let $U$ be a Zariski open dense subset of $X$ and $$ M = (\prod_{v\in S} X(k_v) )\times (\prod_{v\not\in S} M_v )$$  be an open subset of $X(\A_k)$ with $M\cap X(\A_k)^{\Br_1(X)}\neq \emptyset$. If $$(x_v)\in M\cap X(\A_k)^{\Br_1(X)} ,$$  then
 $$i_{y_0}((x_v))\in [M\times Y(\A_k)]\cap [(X\times_k Y)(\A_k)^{\Br_1(X\times_k Y)}] $$ by the functoriality of \BM pairing.
 Since $U\times_k Y$ is a Zariski open dense subset of $X\times_k Y$ and $M\times Y(\A_k)$ is an open subset of $(X\times_k Y)(\A_k)$, one concludes that
 $$ [(U\times_k Y)(k)]\cap [M\times Y(\A_k)] \neq \emptyset $$ by the assumption. This implies that $U(k)\cap M\neq \emptyset$ as required. Similarly, one can prove that $Y$ satisfies Zariski open strong approximation with respect to $\Br_1(Y)$ off $S$ as well.
\end{proof}

\begin{rem} Proposition \ref{product} is also true if one replaces $\Br_1$ with $\Br$ by the same proof.
\end{rem}

The following proposition is a variant of \cite[Lemma 5.1]{CaoXuToric}.

\begin{prop}\label{generalobservation}
Let $f:Y\to X$ be a morphism of algebraic varieties over a number field $k$, $S$ be a finite subset of $\Omega_k$ and $B$ be a subgroup of $\Br(Y)$. 

Suppose $x$ is an isolated point of $X(k)$ inside $X(\A_{k}^{S})$.  For example $X(k)$ is discrete in  $X(\A_{k}^{S})$.  Let $Y_x$ be the the fibre of $f$ over $x$ and $i_x: Y_x\rightarrow Y$ be the corresponding closed immersion. 

If $Y$ satisfies strong approximation with respect to $B$ off $S$, then $Y_{x}$ satisfies strong approximation with respect to $i_x^*(B)$ off $S$.
\end{prop}

\begin{proof} Since $x$ is an isolated point of $X(k)$ inside $X(\A_{k}^{S})$, there is an open neighbourhood $M\subset X(\A_{k}^{S})$ of $x$ such that $X(k)\cap M=\{x\}$.
Let $N$ be an open subset of $Y(\A_{k}^{S})$ such that
$$ [(\prod_{v\in S} Y(k_v))\times N]\cap [ Y_x(\A_{k})^{i_x^*(B)}]\neq \emptyset .$$
Replacing $N$ by $N\cap f^{-1}(M)$ if necessary, the above property still holds. Therefore one can assume $N\subseteq f^{-1}(M)$.
Since $Y$ satisfies strong approximation with respect to $B$ off $S$, there exists $y\in Y(k)\cap N$. This implies that $x=f(y)$ and hence $y\in Y_{x}(k)\cap N$ as required.
\end{proof}

The following corollary is a variant of  \cite[Example 5.2]{CaoXuToric}.

\begin{cor}\label{keylem}
Let $X$ and $Y$ be positive dimensional varieties over a number field $k$ and $S$ be a finite subset of $\Omega_k$. 
 
Let $y\in Y(k)$ such that $\{y\}$ is an isolated point of $Y(k)$ inside $Y(\A_{k}^{S})$. For example, $Y(k)$ is discrete in $Y(\A_{k}^{S})$.

If $V=X\times Y\setminus \{(x,y)\}$ satisfies strong approximation with \BMo off $S$ for some $x\in X(k)$, then $X_{0}=X\setminus \{x\}$ satisfies strong approximation with respect to $\im(\Br(X)\to\Br(X_{0}))$ off $S$.
\end{cor}

\begin{proof}
Consider the following commutative diagram
\SelectTips{eu}{12}$$\xymatrix@C=20pt @R=14pt{
\Br_{}(X\times Y)\ar[r]\ar[d]&\Br_{}(X)\ar[d]
\\  \Br(V)\ar[r]^{i_{y}^{*}}&\Br_{}(X_{0})
}$$ where $X_{0}\xrightarrow{i_{y}} V$ is induced by $y$.  Since the upper horizontal homomorphism of the diagram is surjective and the left vertical homomorphism of the diagram is an isomorphism by the purity theorem for Brauer groups, one concludes that $$i_{y}^{*}(\Br(V))=\im(\Br(X)\to\Br(X_{0})). $$  Restricting the projection $X\times_k Y\to Y$ to $V\to Y$, one obtains the result as desired by Proposition \ref{generalobservation}.
\end{proof}

As a special case of fibration, we study arithmetic purity concerning fibre product by applying the previous results. Arithmetic purity without \BMo is compatible with fibre product, which extends the arithmetic purity of ${\bf A}^n$.

\begin{prop} \label{purity-prod} Let $X$ and $Y$ be smooth and geometrically integral varieties over a number field $k$ and $S$ be a finite subset of $\Omega_k$. Suppose that $X(\A_k)\neq \emptyset$ and $Y(\A_k)\neq \emptyset$. Then $X\times_k Y$ satisfies the arithmetic purity off $S$ if and only if both $X$ and $Y$ satisfy the arithmetic purity off $S$. 
\end{prop}
\begin{proof} Suppose both $X$ and $Y$ satisfy the arithmetic purity off $S$. Let $U$ be an open subset of $X\times_k Y$ with $$\codim((X\times_k Y)\setminus U, X\times_k Y) \geq 2$$ and $p: X\times_k Y \to X$ be the projection map. We apply \citep[Proposition 3.1]{CTXu13} for $U\to p(U)$ to show that $U$ satisfies strong approximation off $S$. Since $p$ is the base change of the structure morphism of $Y$, one has that $p$ is faithfully flat. Therefore $p(U)$ 
is an open subset of $X$ with $$\codim(X\setminus p(U), X)\geq 2 . $$  This implies that $p(U)$ satisfies strong approximation by our assumption and Condition (i) of \citep[Proposition 3.1]{CTXu13} holds. By Proposition \ref{fiber-codim}, there is an open dense subset $V$ of $p(U)$ such that $$\codim(p^{-1}(x)\setminus U, p^{-1}(x))\geq 2$$ for all $x\in V(k)$. Therefore Condition (ii) of \citep[Proposition 3.1]{CTXu13} holds by our assumption for $Y$. Since $U\cap p^{-1}(\xi)\neq \emptyset$ for any $\xi\in p(U)(k_v)$ and $p^{-1}(\xi)(k_v)\neq \emptyset$ with $v\in S$ by our assumption, one concludes that 
$$(U\cap p^{-1}(\xi))(k_v)\neq \emptyset$$ by \cite[Lemma 3.2]{PR}. This implies that $p(U)(k_v)=p(U(k_v))$ for all $v\in S$ and Condition (iii) of \citep[Proposition 3.1]{CTXu13} holds. 

Conversely, suppose $X\times_k Y$ satisfies the arithmetic purity off $S$. Let $U$ be an open subset of $X$ with $\codim(X\setminus U, X)\geq 2$. Then $U\times_k Y$ is an open subset of $X\times_k Y$ with $$\codim((X\times_k Y)\setminus (U\times_k Y), X\times_k Y)\geq 2 . $$ Therefore $U\times_k Y$ satisfies strong approximation off $S$ by our assumption. This implies that $U$ satisfies strong approximation off $S$ by \citep[Proposition 7.1 (2)]{PR}. One concludes that $X$ satisfies the arithmetic purity off $S$. By the same argument, one obtains that $Y$ satisfies the arithmetic purity off $S$ as well.  
\end{proof}

The following result gives a partial answer to Harari's question in \cite[Problem 6]{AIM14}.

\begin{cor} If $G$ is a simply connected linear algebraic group such that $G^{\ss}$ is quasi-split, then $G$ satisfies the arithmetic purity off any finite non-empty subset $S$ of $\Omega_k$. 
\end{cor}
\begin{proof} It follows from \cite[Theorem 2.3]{PR}, \cite[Proposition 3.6]{CaoXuToric}, Theorem \ref{sss-quasi-split} and Proposition \ref{purity-prod}. 
\end{proof}

One can further extend Proposition \ref{purity-prod} to the cases of arithmetic purity with \BM obstruction as follows.

\begin{prop}\label{purity-homotopy} Let $X$ and $Y$ be smooth quasi-projective geometrically integral varieties over a number field $k$, $S$ be a finite subset of $\Omega_k$, $c$ be an integer bigger than 1 and $B$ be a finite subgroup of $\Br(X)$. Assume that $Y$ satisfies the arithmetic purity off $S$ of codimension $c$ with $Y(\A_k)\neq \emptyset$. Then $X$ satisfies the arithmetic purity with respect to $B$ off $S$ of codimension $c$ if and only if $X\times_k Y$ satisfies the arithmetic purity with respect to $p^*(B)$ off $S$ of codimension $c$, where $p: X\times_k Y \to X$ is the projection map.
\end{prop}
\begin{proof}
Suppose $X$ satisfies the arithmetic purity with respect to $B$ off $S$ of codimension $c$. Let $U$ be an open subset of $X\times_k Y$ with $$\codim((X\times_k Y)\setminus U, X\times_k Y) \geq c. $$  Since $p$ is faithfully flat, one obtains that $p(U)$ 
is an open subset of $X$ with $\codim(X\setminus p(U), X)\geq c$.  This implies that $p(U)$ satisfies strong approximation with respect to $B$ off $S$ by our assumption about $X$. Since $Y(\A_k)\neq \emptyset$, one has $p(U)(k_v)=p(U(k_v))$ for all $v\in S$ by \cite[Lemma 3.2]{PR}. By Proposition \ref{fiber-codim}, there is an open dense subset $V$ of $p(U)$ such that $$\codim(p^{-1}(x)\setminus U, p^{-1}(x))\geq c$$ for all $x\in V(k)$. Since $B$ is finite, one obtains that $V(k)$ is dense in $pr^{S}(p(U)(\A_k)^B)$ by Proposition \ref{ndiff}. By Proposition \ref{fibration1} and the assumption about $Y$, one concludes that $U$ satisfies strong approximation with respect to $p^*(B)$ off $S$ as desired. 
 
Conversely, suppose $X\times_k Y$ satisfies the arithmetic purity with respect to $p^*(B)$ off $S$ of codimension $c$. Let $U$ be an open subset of $X$ with $$\codim(X\setminus U, X)\geq c \ \ \ \text{and} \ \ \  M = (\prod_{v\in S} U(k_v) )\times (\prod_{v\not\in S} M_v )$$  be an open subset of $U(\A_k)$ with $M\cap U(\A_k)^{B}\neq \emptyset$. Since $Y(k)\neq \emptyset$, one obtains a section $i_{y_0}: U \to U\times_k Y$ of $p$ by fixing $y_0\in Y(k)$. Then
$$ (M\times Y(\A_k))\cap (U(\A_k)\times Y(\A_k))^{p^*(B)}\neq \emptyset $$  
by functoriality of \BM pairing and $i_{y_0}^{*}\circ p^*=id$. Since $U\times_k Y$ is an open subset of $X\times_k Y$ with $$\codim((X\times_k Y)\setminus (U\times_k Y), X\times_k Y)\geq c,$$ one gets that $U\times_k Y$ satisfies strong approximation with respect to $p^*(B)$ off $S$ by our assumption. This implies that $$(U\times_k Y)(k) \cap (M\times Y(\A_k)) \neq \emptyset. $$ In particular, one has $U(k)\cap M\neq \emptyset$ as desired.
\end{proof}

\begin{cor} \label{purity-prod-zo} Let $X$ and $Y$ be smooth quasi-projective geometrically integral varieties over a number field $k$, $S$ be a finite subset of $\Omega_k$ and $B$ be a subgroup of $\Br(X)$. If $X$ satisfies Zariski open strong approximation with respect to $B$ off $S$ and $Y$ satisfies the arithmetic purity off $S$ of codimension $\dim(X)+1$, then $X\times_k Y$ satisfies the arithmetic purity with respect to $p^*(B)$ off $S$ of codimension $\dim(X)+1$, where $p: X\times_k Y \to X$ is the projection map.
\end{cor}

\begin{proof} Without loss of generality, one can assume that $Y(\A_k) \neq \emptyset$. If $B$ is finite, this is a special case of Proposition \ref{purity-homotopy}. The finiteness of $B$ in the proof of Proposition \ref{purity-homotopy} is only used for the equivalence between strong approximation with respect to $B$ off $S$ with Zariski open strong approximation with respect to $B$ off $S$ by Proposition \ref{ndiff}. For this special case, in the proof of Proposition \ref{purity-homotopy}, one has $p(U)=X$ which satisfies Zariski open strong approximation with respect to $B$ off $S$ by our assumption.   
\end{proof}

Since $X_k=\G_{m}\times_k\G_a\setminus \{(1,0)\}$ does not satisfy strong approximation with respect to $\Br_1(X_k)$ off $\infty_{k}$ for $k=\Q$ or an imaginary quadratic field by \cite[Example 5.2]{CaoXuToric}, the assumption of Zariski open strong approximation with respect to $B$ off $S$ about $X$ in Corollary \ref{purity-prod-zo} is not redundant.

The following lemma is well known. We provide the proof for completeness. 

\begin{lem}\label{open-surj} Let $f:Y\to X$ be a faithfully flat morphism of relative dimension $d$ between geometrically integral varieties over a field $k$.
If $V$ is an open subset of $Y$ such that the dimension of any irreducible component of $Y\setminus V$ is less than d, then the restriction map $f|_{V}: V\to X$ is also faithfully flat.
\end{lem}
\begin{proof} For any geometrical point $x\in X(\bar{k})$,  the dimension of any irreducible component of the fibre $Y_x$ over $x$ is equal to $d$ by the faithful flatness of $f$. Since
the dimension of any irreducible component of $Y\setminus V$ is less than $d$, one concludes that
$$ Y_{x} \cap V_{\bar k} = Y_{x} \setminus (Y_{\bar k} \setminus V_{\bar k}) \neq \emptyset $$ as required.  \end{proof}

Without assuming Zariski open strong approximation, one has the following result.

\begin{prop} \label{purity-no-zo}Let $X$ and $Y$ be smooth quasi-projective geometrically integral varieties over a number field $k$, $S$ be a finite subset of $\Omega_k$, $c$ be an integer bigger than 1 and $B$ be a subgroup of $\Br(X)$. If $X$ satisfies strong approximation with respect to $B$ off $S$ and $Y$ satisfies the arithmetic purity off $S$ of codimension $c$, then $X\times_k Y$ satisfies the arithmetic purity with respect to $p^*(B)$ off $S$ of codimension $c+\dim(X)$, where $p: X\times_k Y \to X$ is the projection map.
\end{prop}

\begin{proof} Without loss of generality, we can assume that $Y(\A_k)\neq \emptyset$. Let $U$ be an open subset of $X\times_k Y$ with $$\codim((X\times_k Y)\setminus U, X\times_k Y)\geq c+\dim (X) . $$ Consider an open subset of $U(\A_k)$  $$ N=\prod_{v\in S} U(k_v) \times \prod_{v\not\in S} N_v  \ \ \ \text{with} \ \ \  N\cap U(\A_k)^{p^*(B)}\neq \emptyset . $$ 
Since $p_U: U\to X$ is surjective by Lemma \ref{open-surj}, one has $p(N)$ is an open subset of $X(\A_k)$ by (\ref{sur-int}) in Proposition \ref{fibration1}. 
The functoriality of Brauer\textendash Manin pairing implies that $p(N)\cap X(\A_k)^B \neq \emptyset $. By \cite[Lemma 3.2]{PR} and $Y(k_v)\neq \emptyset$ for all $v\in S$, one obtains that $p(U(k_v))=X(k_v)$ for all $v\in S$. There is $x\in X(k)\cap p(N)$ by assumption on $X$.  

Since $$ \codim (p^{-1}(x)\setminus U, p^{-1}(x))\geq \dim (Y) -\dim((X\times_k Y)\setminus U)\geq c , $$ one concludes that $p_U^{-1}(x)$ satisfies strong approximation off $S$ by assumption on $Y$. Since $N\cap p_U^{-1}(x)(\A_k) \neq \emptyset$, there is  $$z\in [N\cap (p_U^{-1}(x)(k))]\subseteq N\cap U(k)$$ as desired. \end{proof}

\section{Zariski open strong approximation with \BMo for a connected linear algebraic group} \label{Zosa}

The main results of this section are two folds. The first part is Theorem \ref{iso} which provides the most general descent relation with \BMo for connected linear algebraic groups. 
The second part is Proposition \ref{gclag-zosabm}, Theorem \ref{zosat} and Corollary \ref{Gm}, 
which give a complete answer to Zariski open strong approximation with \BMo for a connected linear algebraic group. 
These results will be used to establish arithmetic purity for connected groups in Section \ref{conngroups} and for partial equivariant smooth compactifications of homogeneous spaces in Section \ref{psc-homog}.

We first improve \cite[Proposition 3.3]{CDX} as follows. 
\begin{thm}\label{iso} If $H\xrightarrow {\phi} G$ is a surjective homomorphism of connected linear algebraic groups over a number field $k$, then  
$$ G(\A_k)^{\Br_1(G)} = \phi(H(\A_k)^{\Br_1 (H)})  \cdot G(k) . $$
\end{thm}
\begin{proof}  We prove this result in three steps. Let $J=\ker(\phi)$. 

Step 0. The result is true when $J$ is connected by \citep[Lemma 3.6, Corollary 3.11 (1) and Corollary 5.13]{Cao-homog}.

Step 1. We prove the result when both $H$ and $G$ are tori and $J$ is finite. By \cite[Theorem 2]{Harari08} and \cite[Theorem 4.10 in Chapter I]{MilneADT}, one has the following commutative diagram of topological groups, with exact columns and rows 
\SelectTips{eu}{12}$$\xymatrix@C=20pt @R=24pt{
& J(\A_k) \ar[r]\ar[d] & H^2(k, J^*)^D  \ar[r] \ar[d]  & \sha^1(k, J) \\
\overline{ H(k)\cdot H(k_\infty)^0} \ar[r]\ar[d] & H(\A_k) \ar[r]\ar[d]^{\phi} & H^2(k, H^*)^D\ar[d] \\
\overline{G(k)\cdot G(k_\infty)^0}  \ar[r] & G(\A_k) \ar[r] \ar[d] & H^2(k, G^*)^D \\
 & \bigoplus_{v\in \Omega_k} H^1(k_v, J)  }$$
 where $H(k_\infty)^0$ and $G(k_\infty)^0$ are the connected components of the Lie groups $H(k_\infty)$ and $G(k_\infty)$ respectively. The result follows from the same 
arguments as those in the proof of \cite[Proposition 3.3]{CDX} with $$\Br_a(H)=H^2(k, H^*) \ \ \ \text{and} \ \ \ \Br_a(G)=H^2(k, G^*) . $$

Step 2.  We prove the result when both $H$ and $G$ are connected and $J$ is finite. 
Since the action of $H$ by conjugation on $J$ is trivial by connectedness of $H$, one obtains that $J$ is a group of multiplicative type. 
By \cite[Proposition 1.3]{CTS87}, we have a short exact sequence  $$ 1\rightarrow J \xrightarrow{\delta} T_2 \xrightarrow{\chi} T_3 \rightarrow 1$$  where 
$T_2$ is a quasi-trivial torus and $T_3$ is a coflasque torus over $k$. 
Let $E_2=H\times_k^{J} T_2$. Then $E_2$ is a connected linear algebraic group such that the following diagram commutes with exact columns and rows  
\SelectTips{eu}{12} \begin{equation}\label{diagram} 
\xymatrix@C=20pt @R=24pt{
& 1 \ar[d] & 1 \ar[d] \\
1\ar[r] & J \ar[r]^{\delta}\ar[d] & T_2 \ar[r]^{\chi}\ar[d]^{\psi_0} & T_3 \ar[d]^{id} \ar[r]&  1 \\
1\ar[r] & H  \ar[r]^{\delta_0}\ar[d]^{\phi} & E_2 \ar[r]^{\chi_0}\ar[d]^{\phi_0} & T_3 \ar[r] &1 \\
& G  \ar[r]^{id}\ar[d]  & G \ar[d]  \\
& 1 & 1 } 
\end{equation} 
where $\delta_0$ and $\psi_0$ are induced by 
$$H\rightarrow H\times_k T_2; \ \ \ g\mapsto (g,1) \ \ \ \text{and} \ \ \ T_2 \rightarrow H\times_k T_2; \ \ \ t\mapsto (1,t) $$
respectively and $\phi_0$ and $\chi_0$ are induced by $\phi$ and $\chi$ with projections respectively. Since $T_3$ is coflasque, one has
$$H^3(k, T_3^*) \cong \prod_{v\in\infty_k} H^3(k_v, T_3^*) \cong \prod_{v\in \infty_k} H^1(k_v, T_3^*) =0 $$ by \citep[Lemma 5.4]{Cao-homog}. 
For any $t\in T_3(k)$, the restriction map
\begin{equation}\label{br}
\Br_1(E_2) \rightarrow \Br_1(\chi_0^{-1}(t)) 
\end{equation} is surjective by \citep[Lemma 5.5]{Cao-homog} and \citep[Lemma 2.7]{CDX}. 

We claim that  \begin{equation}\label{equ} 
 E_2(\A_k)^{\Br_1(E_2)} = \delta_0(H(\A_k)^{\Br_1(H)}) \cdot \psi_0(T_2(\A_k)^{\Br_1(T_2)})\cdot E_2(k) .
\end{equation} 
Indeed, for any $(x_v)\in E_2(\A_k)^{\Br_1(E_2)}$,  one has 
$$ \chi_0((x_v)) \in T_3(\A_k)^{\Br_1(T_3)} = \chi(T_2(\A_k)^{\Br_1(T_2)}) \cdot T_3(k) $$ by Step 1.  There are $(\alpha_v) \in T_2(\A_k)^{\Br_1(T_2)}$ and $t\in T_3(k)$ such that 
$$\chi_0((x_v)) =\chi ((\alpha_v)) \cdot t .$$ This implies that 
$$ \psi_0((\alpha_v)^{-1}) \cdot  (x_v) \in (\chi_0^{-1}(t)(\A_k)) \cap E_2(\A_k)^{\Br_1(E_2)} =\chi_0^{-1}(t)(\A_k)^{\Br_1(\chi_0^{-1}(t))} $$ by (\ref{diagram}) and (\ref{br}). Therefore $\chi_0^{-1}(t)$ is a trivial torsor under $H$ over $k$ by \citep[Theorem 5.2.1]{Skbook}. There is $\eta \in E_2(k)$ such that $t=\chi_0(\eta)$. Since $$(x_v) \cdot \eta^{-1} \cdot \psi_0((\alpha_v))^{-1} \in (\ker(\chi_0)) \cap E_2(\A_k)^{\Br_1(E_2)} =\delta_0(H(\A_k)^{\Br_1(H)})  $$ by (\ref{diagram}) and (\ref{br}),
one obtains that 
$$ (x_v) = [(x_v) \cdot \eta^{-1} \cdot \psi_0((\alpha_v))^{-1} ] \cdot \psi_0((\alpha_v)) \cdot \eta $$ and the claim follows. 

Since 
$$ G(\A_k)^{\Br_1(G)} = \phi_0(E_2(\A_k)^{\Br_1 (E_2)})  \cdot G(k) $$ by Step 0, one concludes that 
$$ G(\A_k)^{\Br_1(G)} = \phi(H(\A_k)^{\Br_1 (H)})  \cdot G(k) $$
as desired by (\ref{diagram}) and (\ref{equ}).

Step 3.  In general, one has a short exact sequence $$1\rightarrow J^0 \rightarrow J \rightarrow \pi_0(J) \rightarrow 1 $$ where $J^0$ is the connected component of identity of $J$ and $\pi_0(J)$ is finite. Since $J^0$ is also a normal subgroup of $H$, one obtains the following two short exact sequences
$$ 1\rightarrow J^0  \rightarrow H \xrightarrow{\phi_1} (H/J^0) \rightarrow 1 $$  
and 
$$ 1\rightarrow \pi_0(J) \rightarrow (H/J^0) \xrightarrow{\phi_2} G \rightarrow 1 $$ with $\phi=\phi_2\circ\phi_1$. 
Then 
\begin{equation}\label{dec1}
 (H/J^0)(\A_k)^{\Br_1(H/J^0)} = \phi_1(H(\A_k)^{\Br_1(H)})\cdot  (H/J^0)(k)  
\end{equation}
by Step 0  and 
\begin{equation} \label{dec2}
 G(\A_k)^{\Br_1(G)} = \phi_2 ( (H/J^0)(\A_k)^{\Br_1(H/J^0)}) \cdot G(k)
\end{equation}  by Step 2. The result follows from combining (\ref{dec1}) and (\ref{dec2}). 
\end{proof}

Zariski open strong approximation with \BMo satisfies the following descent property.

\begin{prop} \label{osatr} Let $H\xrightarrow {\phi} G$ be a surjective homomorphism of connected linear algebraic groups over a number field $k$. If $H$ satisfies Zariski open strong approximation with respect to $\Br_1(H)$ off some finite subset $S$ of $\Omega_k$, then $G$ satisfies Zariski open strong approximation with respect to $\Br_1(G)$ off $S$.
\end{prop}
\begin{proof} 
Let $U$ be a Zariski open dense subset of $G$ and $$V=(\prod_{v\in S}G(k_v))\times \prod_{v\not\in S} V_v $$ be an open subset of $G(\A_k)$  satisfying $V\cap (G(\A_k)^{\Br_1(G)})\neq \emptyset $. Then
there are $g\in G(k)$ and $\alpha\in H(\A_k)^{\Br_1(H)}$ such that $\phi(\alpha) \cdot g \in V$ by Theorem \ref{iso}. Therefore $$\phi^{-1}(V \cdot g^{-1}) \cap (H(\A_k)^{\Br_1 (H)}) \neq \emptyset . $$
Since $H$ satisfies Zariski open strong approximation with respect to $\Br_1(H)$ off $S$, one concludes that
$$ [\phi^{-1}(U\cdot g^{-1})(k)] \cap [\phi^{-1}(V\cdot g^{-1})] \neq \emptyset . $$ This implies that $U(k) \cap V\neq \emptyset$ as required.
\end{proof}

\begin{cor} \label{tori-isogeny} Suppose two tori $T_1$ and $T_2$ are isogenous over a number field $k$. Then $T_1$ satisfies Zariski open strong approximation with respect to $\Br_1(T_1)$ off $\infty_k$ if and only if $T_2$ satisfies Zariski open strong approximation with respect to $\Br_1(T_2)$ off $\infty_k$.
\end{cor}
\begin{proof}
It follows from Proposition \ref{osatr}.
\end{proof}

For a connected linear algebraic group, one can reduce Zariski open strong approximation with \BMo  to its toric part.

\begin{prop} \label{gclag-zosabm} Let $G$ be a connected linear algebraic group over a number field. Assume that $G$ satisfies strong approximation with respect to $\Br_1(G)$ off $\infty_k$. Then $G$ satisfies Zariski open strong approximation with respect to $\Br_1(G)$ off $\infty_k$ if and only if $G^{\textup{tor}}$ satisfies Zariski open strong approximation with respect to $\Br_1(G^{\textup{tor}})$ off $\infty_k$.
\end{prop}
\begin{proof} By Proposition \ref{osatr}, one only needs to show that $G$ satisfies Zariski open strong approximation with respect to $\Br_1(G)$ off $\infty_k$ if $G^{\textup{tor}}$ satisfies Zariski open strong approximation with respect to $\Br_1(G^{\textup{tor}})$ off $\infty_k$.

Since the Zariski open strong approximation property has nothing to do with group structure, we can simply assume that $$G=G^{\textup{u}} \times_k G^{\textup{red}}$$ by \citep[Theorem 2.3]{PR}.
Let $R(G^{\red})$ be the radical of $G^{\red}$. Then $R(G^{\red})$ is a torus by \citep[Theorem 2.4]{PR}. Since 
$$|\Pic((G^{\red}/R(G^{\red}))_{\bar k})| < \infty  \ \ \ \text{and} \ \ \  \bar{k}[G^{\red}/R(G^{\red})]^\times =\bar{k}^\times$$ by \citep[\S0]{CT08}, one concludes that $R(G^{\red})$ is isogenous to $G^{\textup{tor}}$ by \citep[Corollary 6.11]{Sansuc}. By Corollary \ref{tori-isogeny}, one has that $R(G^{\red})$ satisfies Zariski open strong approximation with respect to $\Br_1(R(G^{\red}))$ off $\infty_k$.

Since $G^{\red}$ is connected, the natural homomorphism
$$ \phi: G^{\textup{sc}}\times_k R(G^{\red}) \rightarrow G^{\red}; \ \ \ (x, y)\mapsto \pi(x) \cdot y $$ is surjective with a finite kernel by \citep[Theorem 2.4]{PR}, where $\pi: G^{\textup{sc}}\to G^{\textup{ss}}$ is a simply connected covering over $k$.
By Proposition \ref{osatr}, we only need to show the result for the special case
$ G=G^{\u}\times_k R(G^{\red})\times_k G^{\textup{sc}} $. 

Since $G$ satisfies strong approximation with respect to $\Br_1(G)$ off $\infty_k$, one concludes that $G^{\textup{sc}}$ satisfies strong approximation off $\infty_k$ by \citep[Corollary 5.3]{CDX}. The result follows from Proposition \ref{product} and Proposition \ref{ndiff}.
\end{proof}

\begin{defn} A torus $T$ over a number field $k$ is called \emph{simple} if $T$ contains no non-trivial closed sub-tori over $k$.
\end{defn}

It is clear that a torus $T$ is simple if and only if $T^*\otimes_\Bbb Z \Bbb Q$ is an irreducible Galois module.

\begin{thm} \label{zosat} Suppose $T$ is a simple torus over a number field $k$. Then $T$ satisfies Zariski open strong approximation with respect to $\Br_1(T)$ off $\infty_k$ if and only if $T(k)$ is not discrete in
$T(\A_k^{\infty_k})$.
\end{thm}
\begin{proof} If  $T(k)$ is discrete in $T(\A_k^{\infty_k})$, there is an open compact subgroup $C$ of $T(\A_k^{\infty_k})$ such that $C\cap T(k)=\{1_{T}\}$. Let $U=T\setminus \{1_T\}$. Then $U$ is a Zariski dense open subset of $T$ such that 
$$ [(\prod_{v\in\infty_k} T(k_v)) \times C]\cap (T(\A_k)^{\Br_1 (T)}) \neq \emptyset \ \ \ \text{and} \ \ \ U(k)\cap C =\emptyset .$$
This means that $T$ does not satisfy Zariski open strong approximation with respect to $\Br_1(T)$ off $\infty_k$.

Otherwise, one concludes that $T(k)\cap N$ is infinite for any open compact subgroup $N$ of $T(\A_k^{\infty_k})$. For any open subset $M$ of $T(\A_k)$ with $$M\cap (T(\A_k)^{\Br_1 (T)}) \neq \emptyset , $$ there are $t\in T(k)$ and an open compact subgroup $N_1$ of $T(\A_k^{\infty_k})$ such that
$$ t\cdot N_1\subseteq pr^{\infty_k} (M) $$ by \citep[Theorem 2]{Harari08}. In particular, $T(k)\cap N_1$ is infinite.

Let $U$ be a Zariski open dense subset of $T$ and $V=T\setminus t^{-1}\cdot U$. Since $T$ is simple, one has $V(k)\cap N_1$ is finite by \cite[Theorem 0.2]{Voj96}.
This implies that $(t^{-1}\cdot U)(k)\cap N_1\neq \emptyset$. One concludes that
 $$U(k) \cap (t\cdot N_1) \neq \emptyset \ \   \Rightarrow \ \ U(k) \cap (pr^{\infty_k} (M)) \neq \emptyset$$ as required.
\end{proof}

A torus $T$ satisfying that $T(k)$ is discrete in $T(\A_k^{\infty_k})$ can be determined explicitly by \citep[Theorem 3.5]{LiuXu15}. For example, when $T=\Bbb G_m$, the discreteness is equivalent to $k=\Bbb Q$ or an imaginary quadratic field.
Since any torus $T$ over $k$ is isogenous to the product of simple tori by Maschke's Theorem, one can obtain a complete description of Zariski open strong approximation with \BMo off $\infty_k$ for general tori by Corollary \ref{osatr}.

\begin{cor} \label{Gm} Suppose a torus $T$ is isogenous to $\prod_{i=1}^n T_i$ over a number field $k$ where $T_i$ is a simple torus over $k$ for $1\leq i\leq n$. Then $T$ satisfies Zariski open strong approximation with respect to $\Br_1(T)$ off $\infty_k$ if and only if  $T_i$ satisfies Zariski open strong approximation with respect to $\Br_1(T_i)$ off $\infty_k$ for $1\leq i\leq n$.

In particular, $\Bbb G_m^r$ with some positive integer $r$ satisfies Zariski open strong approximation with \BMo off $\infty_k$ if and only if $k$ is neither $\Bbb Q$ nor an imaginary quadratic field.
\end{cor}

\begin{proof} The first part follows from Corollary \ref{osatr} and Proposition \ref{product}. The second part follows from \citep[Theorem 3.5]{LiuXu15}.
\end{proof}


\begin{rem}\label{Br_1=Br} If $G$ is a connected linear algebraic group, then 
$$ G(\A_k)^{\Br_1(G)} =G(\A_k)^{\Br(G)} $$ by \citep[Remark 5.4]{CDX}. Therefore Theorem \ref{iso}, Proposition \ref{osatr}, Corollary \ref{tori-isogeny}, Proposition \ref{gclag-zosabm}, Theorem \ref{zosat} and Corollary \ref{Gm} are also true with respect to $\Br(\cdot)$.
\end{rem}

\begin{rem} \label{refine-def} Let $X$ be an algebraic variety over a number field $k$ and $B$ be a subgroup of $\Br(X)$. One can refine the definition of strong approximation with respect to $B$ off $\infty_k$ as follows. 

We say $X$ satisfies strong approximation with respect to $B$ off $\infty_k$ if $X(k)$ is dense in $$ X(\A_k)_{\bullet}^B=(\prod_{v\in \infty_k} \pi_0(X(k_v)) \times \prod_{v\not\in \infty_k} X(\A_k^{\infty_k}))^B  $$ where $\pi_0(X(k_v))$ is the set of connected components of $X(k_v)$ for each $v\in \infty_k$. 

We say $X$ satisfies Zariski open strong approximation with respect to $B$ off $\infty_k$ if $U(k)$ is dense in $ X(\A_k)_{\bullet}^B$ for any open dense subset $U$ of $X$.

If $X$ and $Y$ are algebraic varieties over a number field $k$, then $$\pi_0(X(k_v))\times \pi_0(Y(k_v)) =\pi_0 [(X\times_k Y)(k_v)] $$ for any $v\in \infty_k$.

If $H\xrightarrow{\phi} G$ is a surjective homomorphism of connected linear algebraic groups, then $[G(k_v): \phi(H(k_v))]$ is finite by \citep[Theorem 6.14]{PR} for all $v\in \infty_k$. Since the image
of a connected component of $H(k_v)$ under $\phi$ is a connected component of $G(k_v)$ for any $v\in \infty_k$, one concludes that $\phi$ induces a surjective map
$\pi_0(H(k_v))\to \pi_0(G(k_v))$. 

Therefore Proposition \ref{product}, Proposition \ref{osatr}, Corollary \ref{tori-isogeny}, Theorem \ref{zosat}, Corollary \ref{Gm} and Proposition \ref{gclag-zosabm} still hold in this refined sense by the same argument.
\end{rem}

\section{Arithmetic purity with \BMo for connected linear algebraic groups}\label{conngroups}

In this section, we will study arithmetic purity with \BMo for a connected linear algebraic group.  

The arithmetic purity with \BMo for connected linear algebraic groups satisfies the following descent property.

\begin{prop}\label{purity-descent} Let $\phi: H\to G$ be a surjective homomorphism of connected linear algebraic groups over a number field $k$, $S$ be a non-empty finite subset of $\Omega_k$ and $c$ be an integer bigger than $1$. If $H$ satisfies the arithmetic purity with respect to $\Br_1(H)$ off $S$ of codimension $c$, then $G$ satisfies the arithmetic purity with respect to $\Br_1(G)$ off $S$ of codimension $c$.
\end{prop}

\begin{proof} Let $U$ be an open subset of $G$ with $\codim(G\setminus U, G)\geq c$ and $$N=(\prod_{v\in S} U(k_v) )\times \prod_{v\not\in S} N_v $$ be an open subset of $U(\A_k)$ with $N\cap (U(\A_k)^{\Br_1(U)})\neq \emptyset$. There are $g\in G(k)$ and $(h_v)\in H(\A_k)^{\Br_1(H)}$ such that $\phi((h_v))\cdot g \in N$ by Theorem \ref{iso}. 

Since $\phi$ is faithfully flat, one obtains that $\phi^{-1}(U\cdot g^{-1})$ is an open subset of $H$ with $\codim(H\setminus \phi^{-1}(U\cdot g^{-1}), H)\geq c$.  Therefore $\phi^{-1}(U\cdot g^{-1})$ satisfies strong approximation with respect to $\Br_1(\phi^{-1}(U\cdot g^{-1}))$ off $S$. Since 
$$ (h_v) \in \phi^{-1}(N\cdot g^{-1}) \cap \phi^{-1}(U\cdot g^{-1})(\A_k)^{\Br_1(\phi^{-1}(U\cdot g^{-1}))} , $$ there is $x\in \phi^{-1}(U\cdot g^{-1})(k) \cap \phi^{-1}(N\cdot g^{-1})$. Therefore one concludes 
$$ \phi(x) \cdot g \in U(k) \cap N$$ as desired.
\end{proof}

For a general connected linear algebraic group $G$, one has the following arithmetic purity result. 

\begin{cor}\label{alggp2}
Let $G$ be a connected linear algebraic group over a number field $k$. If $G^{\textup{sc}}$ satisfies the arithmetic purity off $\infty_k$. then $G$ satisfies arithmetic purity with respect to  
$\Br_1(G)$ off $\infty_k$ of codimension $2+\dim (G^{\textup{tor}})$.
\end{cor}
\begin{proof} By Proposition \ref{purity-descent} and the same arguments in Proposition \ref{gclag-zosabm}, one only needs to prove that the result holds for the special case where $$G=G^{\u}\times_k R(G^{\red})\times_k G^{\textup{sc}}$$ where $R(G^{\red})$ is the radical of $G^{\textup{red}}$. Since $R(G^{\red})$ is isogenous to $G^{\tor}$ in the proof of Proposition \ref{gclag-zosabm}, one can further assume that $$G=G^{\u}\times_k G^{\tor}\times_k G^{\textup{sc}}$$ by Proposition \ref{purity-descent}. The result follows from 
Proposition \ref{purity-no-zo}, \cite[Theorem 2]{Harari08},  \cite[Proposition 3.6]{CaoXuToric}, Proposition \ref{purity-prod} and \cite[Lemma 2.1]{CDX}. \end{proof}

If $G^{\textup{tor}}$ satisfies Zariski open strong approximation with Brauer\textendash Manin obstruction, one can improve Corollary \ref{alggp2}.

\begin{cor} \label{purity-zo} Let $G$ be a connected linear algebraic group over a number field $k$ with $\bar{k}[G]^\times \neq \bar k^\times$. If $G^{\textup{tor}}$ satisfies Zariski open strong approximation with respect to $\Br_1(G^{\textup{tor}})$ off $\infty_k$ and $G^{\textup{sc}}$ satisfies the arithmetic purity off $\infty_k$, then $G$ satisfies the arithmetic purity with respect to $\Br_1(G)$ off $\infty_k$ of codimension $1+\dim (G^{\textup{tor}})$.
\end{cor}

\begin{proof} By Proposition \ref{purity-descent} and the same arguments as those in Corollary \ref{alggp2}, one can assume that $$ G=G^{\u}\times_k G^{\tor}\times_k G^{\textup{sc}}. $$ 
The result follows from Corollary \ref{purity-prod-zo}, \cite[Proposition 3.6]{CaoXuToric}, Proposition \ref{purity-prod} and \cite[Lemma 2.1]{CDX}. \end{proof}

The assumption that $G^{\textup{sc}}$ satisfies the arithmetic purity off $S$ in Corollary \ref{alggp2} and Corollary \ref{purity-zo} holds when $G^{\ss}$ is quasi-split by Theorem \ref{sss-quasi-split}. On the other hand, one can extend the construction of \cite[Example 5.2]{CaoXuToric} to explain that the codimension condition in Corollary \ref{alggp2} and the assumption that $G^{\textup{tor}}$ satisfies Zariski open strong approximation in Corollary \ref{purity-zo} are not redundant in general.

\begin{prop}\label{counterexconnectedgroups}
Let $G$  be a connected linear algebraic group defined over a number field $k$ with $\bar{k}[G]^\times\neq \bar{k}^\times$ and $f: G\to G^{\textup{tor}}$ be the canonical surjective homomorphism with $\dim (\ker(f))\geq 1$. Suppose that $G^{\tor}(k)$ is discrete in $G^{\tor}(\A_{k}^{\infty})$ and $G^{\ss}$ is simply connected. 

If $Z$ is a translation of a prime divisor of $\ker(f)$ by an element in $G(k)$, then $G\setminus Z$ does not satisfy strong approximation with \BMo off $\infty_{k}$.
\end{prop}

\begin{proof} Without loss of generality, we can assume that $Z$ is a prime divisor of $\ker(f)$. 
Let $U=G\setminus Z$. Since $$\codim(Z, G)=1+ \dim(G^{\tor})\geq 2 , $$ one has the canonical isomorphism $\Br(G)\cong \Br(U)$ by the cohomological purity. Let $$W=U\cap \ker(f)=\ker(f) \setminus Z .$$ Since the following diagram 
\SelectTips{eu}{12}$$\xymatrix@C=20pt @R=14pt{
W \ar[r]^{j \ \ }\ar[d]_{i} &  \ker(f) \ar[d]
\\  U \ar[r] & G
}$$
commutes, one obtains that $$i^*(\Br(U))\subseteq j^*(\Br(\ker(f)))=\Br(k)$$ by \citep[Lemma 2.1]{CDX} and \citep[Proposition 2.6]{CTXu09} and simple connectedness assumption on $G^{\ss}$. 

Suppose that $U$ satisfies strong approximation with \BMo off $\infty_k$. Since $G^{\tor}(k)$ is discrete in $G^{\tor}(\A_{k}^{\infty})$, one gets that $W$ satisfies strong approximation off $\infty_k$ by Proposition \ref{generalobservation}. 
This implies that $W_{\bar{k}}$ is simply connected  by \cite[Theorem 1]{Minchev} (see also \cite[Proposition 2.2]{Rapinchuk}). 
Since $$\Pic(\ker(f)_{\bar{k}})= 1 \ \ \ \text{and} \ \ \ \bar{k}[\ker(f)]^{\times}/\bar{k}^{\times}=1$$  by \cite[Proposition 6.10, Corollary 6.11]{Sansuc}, one concludes that $\bar{k}[W]^{\times}/\bar{k}^{\times}$ is a free abelian group of rank 1 by the following exact sequence
$$\bar{k}[\ker(f)]^{\times}/\bar{k}^{\times}\to\bar{k}[W]^{\times}/\bar{k}^{\times}\to\textup{Div}_{Z_{\bar{k}}}({\ker(f)}_{\bar{k}})\to\Pic({\ker(f)}_{\bar{k}}) . $$
Hence the homomorphism $\bar{k}[W]^{\times}\xrightarrow{(\cdot)^n} \bar{k}[W]^{\times}$ is not surjective for $n\geq2$. The Kummer sequence implies that $$\H^{1}(W_{\bar{k}},\mu_{n})=\Hom(\pi_{1}(W_{\bar{k}}),\Z/n\Z)$$ is not trivial. This contradicts to the simply connectedness of $W_{\bar k}$. \end{proof}

We summarize the above purity results for $GL_n$. 

\begin{ex}\label{counterexGLn} $GL_n$ satisfies the arithmetic purity with respect to $\Br_1(GL_n)$ off $\infty_k$ of codimension 3 over a number field $k$ (Corollary \ref{alggp2}).

$GL_n$ satisfies the arithmetic purity with respect to $\Br_1(GL_n)$ off $\infty_k$ if and only if $k$ is neither $\Q$ nor an imaginary quadratic number field (Corollary \ref{purity-zo} and Proposition \ref{counterexconnectedgroups}). 
\end{ex}

Now we reduce to study the arithmetic purity with \BMo for tori. We first need the following lemma by refining the argument of Harari\textendash Voloch's example \cite[page 420]{HV10}. 

\begin{lem}\label{Gm-pt}
Let $T$ be a torus over a number field $k$ with $\dim(T)=1$ and $U$ be the complement of a $k$-rational point in $T$. If $T(k)$ is not discrete in $T(\A_k^{\infty_k})$, then $U(k)$ is not dense in $pr^{\infty_k} (U(\A_k)^{\Br(T)})$.
\end{lem}

\begin{proof} Without loss of generality, we can assume that $U=T\setminus\{1_T\}$. Let $\mathcal{T}$ be the connected component of the identity section of the N\'eron lft-model of $T$ over $\O_k$ by \cite[Proposition 6, Chapter 10]{BLR}. Then $\mathcal{T}$ is a smooth group scheme of finite type over $\O_k$ such that $T=\mathcal{T}\times_{\O_k} k$. Since $T(k)$ is not discrete in $T(\A_k^{\infty_k})$, one has $${\rm rk}(T_k) < \sum_{v\in \infty_k} {\rm rk}(T_{k_v})$$ by \cite[Theorem 3.5]{LiuXu15}. Then
there is an element $\epsilon\in \mathcal{T}(\O_k)$ of infinite order by \cite[Theorem 5.12]{PR}. 
Let $$ P= \{v\in \Omega_k\setminus \infty_k:  \epsilon \equiv 1_{\mathcal{T}} \mod v \} \cup \infty_k  \ \ \ \text{and} \ \ \  \mathcal{U}=\mathcal{T}\setminus \{1_{\mathcal{T}}\} $$ where $1_{\mathcal{T}}$ is the identity section. Then $P$ is finite and $$\mathcal{U}(\O_P)=\{x_1=\epsilon,  \cdots, x_n \}$$ is finite by Siegel's Theorem (see \cite[\S 2]{Fal86}). 

By Chebotarev Theorem, one can choose a non-dyadic prime $v_0\in \Omega_k\setminus P$ with degree one over $\Bbb Q$ such that $\mathcal{T}(\O_{v_0})=\O_{v_0}^\times $ and $\epsilon \not\equiv x_i \mod v_0$ for $2\leq i\leq n$. Let $\F_{v_0}$ is the residue field of $\O_k$ at $v_0$.  Then $q=|\F_{v_0}|$ is an odd prime and ${\rm {gcd}}(q(q-1), (q-1)^2+1)=1$.   
Consider the sequence $\{\epsilon^l \}_l$  where $l$ runs over all primes satisfying $$l\equiv (q-1)^2+1  \mod q(q-1) $$ in the compact set $$\prod_{v\in \infty_k} \pi_0(T(k_v)) \times \prod_{v\not\in \infty_k} \mathcal{T}(\O_v) $$ where $\pi_0(T(k_v))$ is the set of connected components of the Lie group $T(k_v)$ for $v\in \infty_k$. Then there exists a subsequence converging to an element $(\eta_{v})$. Since $\epsilon^l \perp \Br(T)$ for all primes $l$, one has $(\eta_v)\perp \Br(T)$.

We claim that $\eta_v \in \mathcal{U}(\O_v)$ for all $v\not\in P$.  Otherwise, $\epsilon^l \equiv 1 \mod v$ holds for infinitely many primes $l$. Then $\epsilon\equiv 1 \mod v$ which is a contradiction. For $v\in P\setminus \infty_k$, we claim that $\eta_v\neq 1$. Indeed, since $\epsilon \neq 1$, there is a positive integer $r$ such that $\epsilon \not \equiv 1 \mod v^r$. Suppose that $\eta_v=1$. Then there are infinitely many primes $l$ such that $\epsilon^l \equiv 1 \mod v^r$. This implies that $\epsilon \equiv 1 \mod v^r$. A contradiction is derived.  Since $U(k_v)$ is dense in $T(k_v)$ for each $v\in \infty_k$, one concludes that $$(\eta_v)_{v\not\in \infty_k} \in pr^{\infty_k} [ (\prod_{v\in P} U(k_v))\times (\prod_{v\not\in P} \mathcal{U}(\O_v))]^{\Br(T)} . $$ 

Suppose $U(k)$ is dense in $pr^{\infty_k} (U(\A_k)^{\Br(T)})$. Then $$\mathcal{U}(\O_P) =\overline{\mathcal{U}(\O_P)}= pr^{\infty_k}[(\prod_{v\in P} U(k_v))\times (\prod_{v\not\in P} \mathcal{U}(\O_v))]^{\Br(T)} .$$ 
Since $\epsilon^l \equiv \epsilon \mod v_0 $ by the choice of $l$, one obtains that $(\eta_v)_{v\not \in \infty_k} \neq x_i$ for $2\leq i\leq n$. Therefore $(\eta_v)_{v\not\in \infty_k}=\epsilon$. Write $\epsilon=\zeta+ a\pi_{v_0}^s$ where $\pi_{v_0}$ is a prime element of $k_{v_0}$, $\zeta$ is a root of unity of order dividing $q-1$, $a\in \O_{v_0}^\times$ and $s$ is a positive integer. Then
$$ \epsilon^{l-1}=(1+a\zeta^{-1} \pi_{v_0}^s)^{l-1}\equiv 1+ a \zeta^{-1} \pi_{v_0}^s \ \mod \pi_{v_0}^{s+1} . $$  
This contradicts that $\epsilon^{l-1} \to 1$ for the primes $l$ in the above convergent subsequence. 
\end{proof}

The following example explains that the arithmetic purity for tori is not true for arbitrarily large codimension in general (see Remark \ref{counter-example-n}).

\begin{ex}\label{torus-pt} Let $F$ be an imaginary quadratic number field and $R^{1}_{F/\Q}\G_{m}$ be a norm one torus over $\Q$. Suppose that $X$ is the complement of a $\Q$-rational point in $R^{1}_{F/\Q}\G_{m}\times_{\Q} \G_{m}$. Then
\begin{itemize}
\item[-] $X_k$ satisfies strong approximation with \BMo off $\infty_{k}$ but does not satisfy Zariski open strong approximation with \BMo off $\infty_k$ if $k=\Q$ or $F$;
\item[-] $X_{k}$  does not satisfy strong approximation with \BMo off $\infty_{k}$ if k is a totally real  number field other than $\Q$ or an imaginary quadratic field other than $F$.
\end{itemize}
\end{ex}

\begin{proof} When $k=\Q$ or $F$, one has that $X_k$ satisfies strong approximation with \BMo off $\infty_{k}$ by \cite[Corollary 3.6 and Example 3.7]{LiuXu15}. In fact, one obtains that $X_k(k)$ is open and closed in $pr^{\infty_k}(X(\A_k))$ by \cite[Theorem 3.5]{LiuXu15}. Let $U=X_k\setminus \{P\}$ with $P\in X_k(k)$. Then $U$ is a Zariski open dense subset of $X$ but $U(k)$ is also open and closed in $pr^{\infty_k}(X(\A_k))$. The first statement follows.

For the second one, one considers the projection $X_k \to Y_k$ where 
$$ Y_k= \begin{cases} R^{1}_{kF/k}(\G_m) \ \ \ & \text{ if $k$ is a totally real number field other than $\Q$ ;} \\
 \G_{m,k} \ \ \ & \text{ if $k$ is an imaginary quadratic field other than $F$.} \end{cases} $$
  The result follows from \cite[Theorem 5.12]{PR}, \cite[Theorem 3.5 or Example 3.7]{LiuXu15}, Corollary \ref{keylem} and Lemma \ref{Gm-pt}.
\end{proof}

\begin{rem} Example \ref{torus-pt} explains that some geometrically rational open surfaces satisfy strong approximation with \BMo over the ground field but fail to satisfy strong approximation with \BMo over some finite extension of the ground field. 
\end{rem}

\begin{rem} \label{counter-example-n} One can also have the counter-examples of tori which do not satisfy arithmetic purity with \BMo of any codimension $n$ by modifying $X=((R^{1}_{F/\Q}\G_{m})^{n-1}\times_{\Q} \G_{m})\setminus\{1\}$ in Example \ref{torus-pt}. Then $X_k$ over any totally real field $k$ other than $\Q$ does not satisfy strong approximation with \BMo off $\infty_k$.  \end{rem}

\begin{rem} As pointed out in Remark \ref{Br_1=Br}, Proposition \ref{purity-descent}, Corollary \ref{alggp2}, Corollary \ref{purity-zo} and Example \ref{counterexGLn}
are also true with respect to $\Br(\cdot)$.
\end{rem}

\section{Arithmetic purity with \BMo for partial equivariant smooth compactifications of homogeneous spaces}\label{psc-homog}

In this section, we give a proof of Theorem \ref{maintheorem} using the previous results and \cite{Cao-homog}. 
The following lemma is an immediate consequence of \cite[Theorem 2.2]{PR}.

\begin{lem}\label{Zariski-dense-image} If $ G_1 \xrightarrow{f} G_2$ is a surjective homomorphism of connected linear algebraic groups over a field $k$ of characteristic 0, then $f(G_1(k))$ is Zariski dense in $G_2$.
\end{lem} 
\begin{proof} Let $W\neq \emptyset$ be an open subset of $G_2$. Since $f$ is surjective, one has $f^{-1}(W)\neq \emptyset$. By \cite[Theorem 2.2]{PR}, one concludes that 
$$ f^{-1}(W) \cap G_1(k) \neq \emptyset . $$ This implies that $f(G_1(k))\cap W\neq \emptyset$ as desired. 
\end{proof}

First, we need the following result.

\begin{prop}\label{hom-inv} Let $G$ be a connected linear algebraic group over a number field $k$ and $S\neq \emptyset$ be a finite subset of $\Omega_k$. Suppose that $H$ is a closed subgroup of $G$ over $k$ such that the induced restriction map $G^* \to H^*$ is injective and $G^{\textup{sc}}$ satisfies the arithmetic purity off $S$. Assume that $U$ is an open subset of $G$ over $k$ such that $ \codim(G\setminus U, G)\geq 2$ and $ U H= U$. If $W$ is an open dense subset of $U$, then $W(k) \cdot H(k_\infty)^0$ is dense in $pr^{S}(U(\A_{k})^{\Br_1(G)})$ where $H(k_\infty)^0$ is the connected Lie subgroup of finite index in $H(k_\infty)$. 
\end{prop}
\begin{proof} We prove this result in the following steps. 

Step 1. We first consider $G=G^{\textup{sc}}\times_k T$, where $G^{\textup{sc}}$ is a semi-simple simply connected linear algebraic group and $T$ is a torus. 

Let $p:  G=G^{\textup{sc}}\times_k T \to T$ be the projection map. Since $G^* =T^* \to H^*$ is injective, the restriction $p_H: H\to T$ is surjective. Since $p$ is faithfully flat, one has $p(U)$ is an open subset of $T$ with $p(U)\cdot T=p(U)$. This implies that $p(U)=T$. 
By Proposition \ref{fiber-codim}, there is an open dense subset $V$ of $T$ such that $$ \codim(p^{-1}(x)\setminus U, p^{-1}(x)) \geq 2$$ for all $x\in V(k)$. 
Let
$$ M= (\prod_{v\in S} U(k_v) ) \times \prod_{v\not\in S} M_v $$ be an open subset of $U(\A_k)$ such that $M \cap U(\A_k)^{\Br_1(G)} \neq \emptyset $. Then $p(M\cdot H(k_\infty)^0)$  is an open subset of $T(\A_k)$ with $$p(M\cdot H(k_\infty)^0)\cap T(\A_k)^{\Br_1(T)} \neq \emptyset$$ by (\ref{sur-int}), functoriality of Brauer-Manin pairing and \cite[Lemma 2.1]{CDX}. Since $p(H(k_\infty)^0)=T(k_\infty)^0$ the connected Lie subgroup of finite index in $T(k_\infty)$, there is $$t_0\in T(k)\cap p(M\cdot H(k_\infty)^0)$$ by \cite[Theorem 2]{Harari08}. 

Let $V_1= p(W) \cap V\neq \emptyset$. There is $h_0\in H(k)$ such that $p(h_0)\in  t_0^{-1}V_1(k)$ by Lemma \ref{Zariski-dense-image}. This implies that $\xi=t_0\cdot p(h_0)\in V_1(k)$. Then 
$$ (M\cdot H(k_\infty)^0 \cdot h_0)\cap p^{-1}(\xi)(\A_k) $$
is a non-empty open subset of $(p^{-1}(\xi)\cap U)(\A_k)$ such that $v$-component of this open subset is $(p^{-1}(\xi)\cap U)(k_v)$ for all $v\in S$. Since $UH=U$, one has that $W\cdot h_0$ is an open dense subset of $U$. Since $G^{\textup{sc}}$ satisfies the arithmetic purity off $S$,  one gets 
$$ [((W\cdot h_0) \cap p^{-1}(\xi))(k)] \cap[(M\cdot H(k_\infty)^0 \cdot h_0)\cap p^{-1}(\xi)(\A_k)] \neq \emptyset  $$ by Proposition \ref{ndiff}.
This implies that $(W(k)\cdot H(k_\infty)^0) \cap M\neq \emptyset$ as desired.

Step 2. We prove that the result holds for $G=G^{\red}$.

By \cite[Theorem 2.4]{PR}, one has a surjective homomorphism $$\pi: \  G^{\textup{sc}}\times_k R(G)  \rightarrow G$$ with a finite kernel, where $G^{\textup{sc}}$ is a simply connected covering of $G^{\ss}$ and $R(G)$ is the solvable radical of $G$. Then the following diagram of exact sequences
 \SelectTips{eu}{12} $$
\xymatrix@C=20pt @R=24pt{
1\ar[r] & G^* \ar[r]\ar[d] & [G^{\textup{sc}} \times_k R(G)]^*  \ar[r]\ar[d] & \ker(\pi)^* \ar[d]^{\cong} \\
1\ar[r] & H^*   \ar[r] & \pi^{-1}(H)^* \ar[r] & \ker(\pi)^* 
}$$ commutes. 
Since the left column in the above diagram is injective, one obtains the middle column in the above diagram is injective as well. 

Let  
$$ M= (\prod_{v\in S} U(k_v) ) \times \prod_{v\not\in S} M_v $$ be an open subset of $U(\A_k)$ such that $M\cap U(\A_k)^{\Br_1(G)} \neq \emptyset$. Take
$$(x_v) \in (M\cap U(\A_k)^{\Br_1(G)} )\subseteq G(\A_k)^{\Br_1(G)} . $$ There are $g\in G(k)$ and $$(y_v)\in [(G^{\textup{sc}}\times_k R(G))(\A_k)]^{\Br_1(G^{\textup{sc}}\times_k R(G))} $$ such that $(x_v)=g\cdot \pi((y_v))$ by Theorem \ref{iso}. This implies that 
$$ \pi^{-1}(g^{-1} M ) \cap [\pi^{-1}(g^{-1} U)(\A_k)]^{\Br_1(\pi^{-1}(g^{-1}U))} \neq \emptyset  . $$
Since $[\pi^{-1}(g^{-1} U)][\pi^{-1}(H)]= \pi^{-1}(g^{-1}U)$ and $$ \codim( (G^{\textup{sc}}\times_k R(G)) \setminus \pi^{-1}(g^{-1}U),  G^{\textup{sc}}\times_k R(G) )\geq 2 , $$ one concludes that $$ [\pi^{-1}(g^{-1}W)(k) \cdot \pi^{-1}(H) (k_\infty)^0] \cap \pi^{-1} (g^{-1}M) \neq \emptyset  $$ by Step 1, where $\pi^{-1}(H) (k_\infty)^0$ is the connected Lie subgroup of finite index in $\pi^{-1}(H) (k_\infty)$. 
Therefore $(W(k)\cdot H(k_\infty)^0) \cap M\neq \emptyset$ as desired. 

Step 3. For a general connected linear algebraic group $G$, one has the following short exact sequence of linear algebraic groups
$$ 1\rightarrow G^{\u} \rightarrow G \xrightarrow{\phi} G^{\red} \rightarrow 1 . $$ Since the following diagram 
\SelectTips{eu}{12} $$
\xymatrix@C=20pt @R=24pt{
 (G^{\red})^* \ar[r]^{\ \ \ \phi^*}_{\ \ \ \cong}\ar[d] & G^* \ar[d]  \\ 
 \phi(H)^*   \ar[r] & H^*}$$
commutes, one obtains that the natural map $(G^{\red})^* \to \phi(H)^*$ is injective by the assumption. Since $\phi$ is faithfully flat, one concludes that $\phi(U)$ is an open subset of $G^{\red}$ with 
$$ \codim(G^{\red}\setminus \phi(U), G^{\red})\geq 2 \ \ \ \text{and} \ \ \ \phi(U) \phi(H)=\phi(U) . $$
By Proposition \ref{fiber-codim}, there is an open dense subset $V$ of $G^{\red}$ such that $$ \codim(\phi^{-1}(x)\setminus U, \phi^{-1}(x)) \geq 2$$ for all $x\in V(k)$. 

Let
$$ M= (\prod_{v\in S} U(k_v) ) \times \prod_{v\not\in S} M_v $$ be an open subset of $U(\A_k)$ such that $M \cap U(\A_k)^{\Br_1(G)} \neq \emptyset $. Then $\phi (M\cdot H(k_\infty)^0)$  is an open subset of $\phi(U)(\A_k)$ with $$\phi (M\cdot H(k_\infty)^0)\cap [\phi(U)(\A_k)^{\Br_1(G^{\red})}] \neq \emptyset$$ by (\ref{sur-int}), functoriality of Brauer-Manin pairing and \cite[Lemma 2.1]{CDX}. Since $H^1(k_v, G^{\u})=\{1\}$ for all $v\in \Omega_k$, one gets that $\phi(U(k_v))=\phi (U)(k_v)$ for all $v\in S$ by \cite[Lemma 3.2]{PR}. Moreover, $\phi (H(k_\infty)^0)= \phi(H)(k_\infty)^0$ the connected Lie subgroup of finite index in $\phi (H)(k_\infty)$. 
There is 
 $$\eta \in (V\cap \phi(W))(k)\cap \phi (M\cdot H(k_\infty)^0)$$ by Step 2. 
 
 Since $$(M\cdot H(k_\infty)^0)\cap \phi^{-1}(\eta)(\A_k) \neq \emptyset \ \ \ \text{and} \ \ \ W\cap \phi^{-1}(\eta)\neq \emptyset , $$ one concludes that 
 $$(W\cap \phi^{-1}(\eta))(k) \cap (M\cdot H(k_\infty)^0) \neq \emptyset $$ as desired by 
\cite[Proposition 3.6]{CaoXuToric} and Proposition \ref{ndiff}.
\end{proof}

Before proving the main results of this section, we recall the definition of the invariant Brauer group of $G$-variety in \cite[Definition 3.1]{Cao-homog}. Let $X$ be a $G$-variety with an action
$G\times_k X \xrightarrow{\rho} X$ over $k$. Define
$$ \Br_G(X)=\{ b\in \Br(X): \ \rho^*(b)-p_X^*(b) \in p_G^*(\Br(G)) \} $$
where $p_X: G\times_k X\rightarrow X$ and $p_G: G\times_k X \rightarrow G$ are the projection maps. When $X$ is a homogeneous space of $G$, one has 
$$ \Br_G(X) = \Br_1(X, G) = \ker ( \Br(X) \to \Br(G_{\bar k})) $$ by \cite[Proposition 3.9]{Cao-homog}. 

\begin{thm}\label{purity-hom} Suppose that $X=G/H$ with $\bar k[X]^\times=\bar k^\times$, where $G$ is a connected linear algebraic group over a number field $k$ and $H$ is a connected closed subgroup of $G$ over $k$. Let $S\neq \emptyset$ be a finite subset of $\Omega_k$. If $G^{\textup{sc}}$ satisfies the arithmetic purity off $S$, then $X$ satisfies the arithmetic purity with respect to $\Br_G(X)$ off $S$.
\end{thm}
\begin{proof}  Let $U$ be an open subset of $X$ such that $\codim(X\setminus U, X)\geq 2$ and $$ M= (\prod_{v\in S} U(k_v) ) \times \prod_{v\not\in S} M_v $$ of $U(\A_k)$ such that $M\cap U(\A_k)^{\Br_G(X)} \neq \emptyset$. There is $\sigma\in \H^1(k, H)$ such that 
 $$ M\cap \pi_\sigma ( G^\sigma(\A_k)^{\Br_1(G^\sigma)}) \neq \emptyset $$ 
by \cite[Corollary 5.11]{Cao-homog}, where $\pi_\sigma: G^\sigma\to X$ is the twist of the quotient map $\pi: G\to X$ by $\sigma$. Since $G^\sigma$ is a left torsor under $G$ over $k$, one has that $G^\sigma$ is a trivial torsor under $G$ over $k$ by \cite[Theorem 5.2.1]{Skbook}. This implies that there is $g\in G(k)$ such that 
$$ g (\pi^{-1}(M)) \cap (g(\pi^{-1}(U))(\A_k))^{\Br_1(G)} \neq \emptyset . $$  
Since $\pi: G\to X$ is faithfully flat, one has $$\codim(G\setminus g(\pi^{-1}(U)), G)\geq 2 . $$ Moreover, the natural map $G^* \to H^*$ is injective by $\bar k[X]^\times =\bar k^\times$ and $$g(\pi^{-1}(U)) H =g (\pi^{-1}(U)) .$$  Therefore 
$$[g(\pi^{-1}(U))(k) \cdot H(k_\infty)^0] \cap g(\pi^{-1}(M) )\neq \emptyset $$ by Proposition \ref{hom-inv}. Therefore $U(k) \cap M\neq \emptyset$ as desired. 
\end{proof}

The following example provides a partial answer to Wittenberg's question in \cite[Problem 6]{AIM14}.

\begin{ex} If $q(x_1,x_2,x_3,x_4)=x_1x_2+x_3x_4$, then the variety $X$ defined by the equation $q(x_1,x_2,x_3,x_4)=c$ with $c\in k^\times$ satisfies arithmetic purity off 
any non-empty finite subset $S$ of $\Omega_k$.
\end{ex}
\begin{proof} Fix a rational point $\delta_0=(1, c, 0, 0)$. Then $X$ is isomorphic to the homogeneous space $Spin(q)/Stab(\delta_0)$ where 
$$Stab(\delta_0)= \{g\in Spin(q): g \delta_0 =\delta_0 \}$$ is a spin group of the three dimensional non-degenerated quadratic space. This implies that  $$\Br(X)=\Br_G(X)=\Br(k)$$ \cite[Proposition 2.6 and Proposition 2.10]{CTXu09} and \cite[Proposition 3.9]{Cao-homog}. 
Since $q$ is split, the result follows from Theorem \ref{purity-hom}.   
\end{proof}

Based on Theorem \ref{purity-hom}, we complete the proof of Theorem \ref{maintheorem} by using the results in \cite{Cao-homog}. 

\begin{thm}\label{purity-com-hom} Suppose $X$ is a smooth and geometrically integral $G$-variety containing a Zariski open dense $G$-orbit $Z\cong G/H$ over a number field $k$, where $G$ is a connected linear algebraic group and $H$ is a connected closed subgroup of $G$ over $k$. Let $S\neq \emptyset$ be a finite subset of $\Omega_k$. If $\bar{k}[X]^\times=\bar{k}^\times$ and $G^{\textup{sc}}$ satisfies the arithmetic purity off $S$, then $X$ satisfies the arithmetic purity with respect to $\Br_G(X)$ off $S$. 
\end{thm}
  
\begin{proof} Let $T$ be a torus over $k$ such that $T^*=\bar{k}[Z]^\times/\bar{k}^\times$. Then there is a surjective morphism $\pi: Z\to T$ over $k$ such that $\pi\circ \phi: G\to T$ is a surjective homomorphism of algebraic groups and $G_0=\ker(\pi\circ \phi)$ is connected by \cite[Lemma 3.21 and Proposition 3.22]{Cao-homog}, where $\phi : G\to G/H$ is the quotient map. This implies that $G_0^{\ss}=G^{\ss}$ and $H\subseteq G_0$. Note that we also have $G_0^{\textup{sc}}=G^{\textup{sc}}$. 

For any $t\in T(k)$, the fibre $\pi^{-1}(t)$ is a homogeneous space of $G_0$ over $k$ satisfying $$\pi^{-1}(t)_{\bar k} \cong (G_0/H)_{\bar k}$$ over $\bar{k}$. Since $\bar{k}[X]^\times=\bar{k}^\times$, one has that $\bar{k}[\pi^{-1}(t)]^\times=\bar{k}^\times$ by \cite[Proposition 3.13]{Cao-homog}. 

Let $U$ be a Zariski open subset of $X$ with $\codim(X\setminus U, X)\geq 2$ and $$ M= (\prod_{v\in S} U(k_v) ) \times \prod_{v\not\in S} M_v $$ be an open subset of $U(\A_k)$ such that $M\cap U(\A_k)^{\Br_G(X)} \neq \emptyset$. Since $$\Br_G(X)= \Br_G(Z)\cap \Br(X) $$ by \cite[Proposition 3.4(3)]{Cao-homog}, there is $\xi \in T(k)$ such that 
$$ \codim(\pi^{-1}(\xi)\setminus U, \pi^{-1}(\xi))\geq 2  \  \ \text{and}  \ \ (\pi^{-1}(\xi)\cap U)(\A_k)^{\Br_{G_0}(\pi^{-1}(\xi)) }\cap M\neq \emptyset  $$ by \cite[Corollary 6.13]{Cao-homog}. Then $\pi^{-1}(\xi)(k)\neq \emptyset$ by \cite[Theorem 5.2.1]{Skbook}. Applying Theorem \ref{purity-hom} to $\pi^{-1}(\xi)$, one concludes that 
$$ U(k)\cap M\supseteq (U\cap \pi^{-1}(\xi))(k) \cap M \neq \emptyset$$ as desired. 
\end{proof}

\begin{rem} The purity assumption on $G^{\textup{sc}}$ in Proposition \ref{hom-inv}, Theorem \ref{purity-hom} and Theorem \ref{purity-com-hom} holds when $G^{\ss}$ is quasi-split by Theorem \ref{sss-quasi-split}. \end{rem}

\section{Examples}\label{examples}
In this section, we are going to produce some examples for which arithmetic purity for strong approximation does not hold. These examples show that geometric assumptions on $\H^{i}(X_{\bar{k}},\G_{m})(i=0,1)$ in Question \ref{mainquestion} are necessary.

In \cite{Harari08}, Harari showed that semi-abelian varieties satisfy strong approximation with \BMo by assuming the finiteness of Tate-Shafarevich groups. In this section we give examples to explain that abelian varieties do not satisfy arithmetic purity with \BMo even for arbitrarily large codimension.

First we generalise Harari\textendash Voloch's example \cite[page 420]{HV10} as follows.

\begin{thm}\label{elliptic}
Let $E$ be an elliptic curve defined over a number field $k$ such that the Mordell\textendash Weil rank of $E(k)$ is positive. If $E_{0}$ is the complement of a $k$-rational point in $E$, then $E_{0}$ does not satisfy strong approximation with \BMo off $\infty_{k}$.
\end{thm}

\begin{proof} Without loss of generality, we may assume that $E_{0}$ is the complement of identity $O\in E(k)$ in $E$. 
Let $\E$ be the N\'eron model of $E$ over $\O_{k}$ and $\E_{0}$ be the complement of identity section in $\E$. 
Fix a rational point $Q\in E(k)$ of infinite order. Then 
 $$P = \{ v\in \Omega_k\setminus \infty_k:  \ Q\equiv O\mod v \} \cup \infty_k $$ is finite. 
By Siegel's Theorem, $\E_{0}(\O_P)=\{x_1=Q, \cdots, x_n\}$ is finite.  

Choose a non-dyadic prime $v_0$ of $k$ with $v_0\notin P$ such that $Q\not\equiv x_i \mod v_0$ for $2\leq i\leq n$ and $E$ has a good reduction at $v_0$. Let $q$ be the characteristic of the residue field $\F_{v_0}$ of $\O_k$ at $v_0$. Since $(|\E(\F_{v_0})|+1)-(|\E(\F_{v_0})|-1)=2$, one obtains that $${\rm gcd} ((|\E(\F_{v_0})|+1), q)=1 \ \ \ \text{or} \ \ \  {\rm gcd} ((|\E(\F_{v_0})|-1), q)=1 . $$ 
Let $$ a_{v_0}=\begin{cases} |\E(\F_{v_0})|+1 \ \ \ & \text{if ${\rm gcd}((|\E(\F_{v_0})|+1), q)=1$} \\
(q-1)|\E(\F_{v_0})|+1 \ \ \ & \text{otherwise.} \end{cases} $$
Let $\pi_0(E(k_v))$ be the set of connected components of the Lie group $E(k_v)$ for $v\in \infty_k$.
Consider a sequence $\{ lQ \}_l$ in the compact set $$\prod_{v\in \infty_k} \pi_0(E(k_v)) \times \prod_{v\notin\infty_{k}}\E(\O_{v}) , $$ where $l$  runs over all primes satisfying $$l \equiv a_{v_0}  \mod q|\E(\F_{v_0})| .$$  There exists a subsequence
converging to an element $(y_{v})_{v}$. 

For $v\not \in P$, one has $y_v\not\equiv O \mod v$. Otherwise, $l Q\equiv O \mod v$ for infinitely many primes $l$. This implies $Q\
\equiv O \mod v$ which is a contradiction. For $v\in P\setminus \infty_k$, one has $y_v \in E_0(k_v)$. Indeed, since $Q$ is not the identity section $O$, there is a positive integer $r$ such that $Q\not\equiv O$ in $\E (\O_v/v^r)$. If $y_v =O$, there are infinitely many primes $l$ such that $l Q \equiv O$ in $\E(\O_v/ v^r)$. This implies $Q\equiv O$ in $\E(\O_v/ v^r)$ and a contradiction is derived. Since $E_0(k_v)$ is dense in $E(k_v)$ for all $v\in \infty_k$, one concludes that 
$$(y_{v})_{v\not\in \infty_k}\in pr^{\infty_k} ([(\prod_{v\in P}E_{0}(k_{v}))\times(\prod_{v\notin P}\E_{0}(\O_{v}))]^{\Br(E)}) . $$ By \citep[Example 2.22(a) in Chapter III]{Milne80} and \citep[Theorem 6.4.4]{GS06}, one has $\Br(E)=\Br(E_0)$. 

Suppose $E_{0}(k)$ is dense in $pr^{\infty_k}(E_0(\A_k)^{\Br(E_0)})$. Then 
$$ \E_0(\O_P) =\overline{\E_0(\O_P)} = pr^{\infty_k}[(\prod_{v\in P}E_{0}(k_{v}))\times(\prod_{v\notin P}\E_{0}(\O_{v}))]^{\Br(E_0)} .$$
Since $l Q\equiv Q \not\equiv x_i \mod v_0$ for $2\leq i\leq n$ by the choice of $l$, one concludes that $Q=(y_v)_{v\not\in \infty_k}$. Namely,  $(l-1)Q\to O$ for the primes $l$ in the above convergent subsequence. On the other hand,  $E(k_{v_0})$ contains a subgroup of finite index which is isomorphic to $(\O_{v_0},+)$ as topological groups by \cite[Theorem 7]{Mattuck}. There exists a positive integer $N$ such that $$N\cdot Q \in \O_{v_0} \ \ \ \text{and} \ \ \ N\cdot Q\neq 0 . $$ This implies that $l\to 1$ in $\O_{v_0}$ for the above primes $l$. A contradiction is derived.
\end{proof}

\begin{cor}\label{abelianvar}
Let $E$ be a semi-abelian variety with $dim(E)=1$ over a number field $k$ such that $E(k)$ is not discrete in $E(\A_k^{\infty_k})$. If $A$ is a non-trivial semi-abelian variety over $k$ such that $A(k)$ is discrete in $A(\A_k^{\infty_k})$, then the complement of a $k$-rational point in $A\times_k E$ does not satisfy strong approximation with \BMo off $\infty_k$.
\end{cor}

\begin{proof}
The result follows from Corollary \ref{keylem} and Theorem \ref{elliptic} and Lemma \ref{Gm-pt}.
\end{proof}

Let us write down some explicit examples. Let $k=\Q(\sqrt{-5})$ and $E_{i}(i=1,2)$ be the elliptic curves defined over $\Q$ by the equations
$$E_{1}:\mbox{ }\mbox{ }\mbox{ }y_{1}^{2}=x_{1}^{3}-x_{1},$$
$$E_{2}:\mbox{ }\mbox{ }\mbox{ }y_{2}^{2}=x_{2}^{3}-4x_{2}.$$
The vanishing orders at $s=1$ of $L$-functions of  $E_i$ and their quadratic twists $E_{i}^{(-5)}$ have been calculated, which are at most $1$, by the work of Tian, Yuan and Zhang \cite{TYZ}. As a consequence of Gross\textendash Zagier's work in \cite{GrossZagier} and Kolyvagin's work in \cite{Kolyvagin}, the Tate\textendash Shafarevich groups $$\sha(E_{i},\Q) \ \ \ \text{ and }  \ \ \  \sha(E_{i}^{(-5)},\Q)$$  are finite for $i=1,2$.  One also deduces that $E_{1}(\Q)$, $E_{2}(\Q)$, and $E_{2}(k)$ are all finite, whereas $E_{1}(k)$ is of rank $1$.  With some more effort, one also deduces the finiteness of $$\sha(E_{i},k)=\sha(R_{k/\Q}E_{i,k},\Q)$$ by  \cite[Example 1]{Milne72} and  \cite[Lemma I.7.1(b)]{MilneADT}.

For any positive integer $n$, the abelian variety $A=E_{2}^{n-1}\times E_{1}$ over $\Q$ satisfies strong approximation with \BMo off $\infty_{\Q}$ by the Cassels\textendash Tate dual exact sequence. Then $A\setminus\{O\}$ satisfies strong approximation with \BMo off $\infty_{\Q}$ by \cite[Proposition 3.1]{LiuXu15}. However, this cannot be preserved after base change to $k$. Indeed, $A_{k}$ satisfies strong approximation with \BMo off $\infty_{k}$, but $A_{k}\setminus\{O\}$ does not satisfy strong approximation with \BMo off $\infty_{k}$ as indicated in Corollary \ref{abelianvar}.

\bigskip

\noindent\textbf{Acknowledgements.} We would like to thank J.-L. Colliot-Th\'el\`ene and Olivier Wittenberg for their useful comments on the original version of this paper. We would also like to thank Philippe Gille and Ye Tian for helpful discussion. The first named author acknowledges the support of the French Agence Nationale de la Recherche (ANR) under reference ANR-12-BL01-0005 and the third named author acknowledges the support of  NSFC grant no.11471219 and no.11631009. 

\bigskip


\bibliographystyle{alpha}

\end{document}